\documentclass[11pt]{article}
\usepackage{latexsym,amssymb,amsmath}
\usepackage[notref,notcite]{showkeys}
\usepackage{amssymb,amsthm,upref,amscd}
\usepackage{color}
\usepackage[T1]{fontenc}
\begin{document}
\baselineskip=15pt

\numberwithin{equation}{section}

\newtheorem{thm}{Theorem}[section]
\newtheorem{lem}[thm]{Lemma}
\newtheorem{cor}[thm]{Corollary}
\newtheorem{Prop}[thm]{Proposition}
\newtheorem{Def}[thm]{Definition}
\newtheorem{Rem}[thm]{Remark}
\newtheorem{Ex}[thm]{Example}

\newcommand{\A}{\mathbb{A}}
\newcommand{\B}{\mathbb{B}}
\newcommand{\C}{\mathbb{C}}
\newcommand{\D}{\mathbb{D}}
\newcommand{\E}{\mathbb{E}}
\newcommand{\F}{\mathbb{F}}
\newcommand{\G}{\mathbb{G}}
\newcommand{\I}{\mathbb{I}}
\newcommand{\J}{\mathbb{J}}
\newcommand{\K}{\mathbb{K}}
\newcommand{\M}{\mathbb{M}}
\newcommand{\N}{\mathbb{N}}
\newcommand{\Q}{\mathbb{Q}}
\newcommand{\R}{\mathbb{R}}
\newcommand{\T}{\mathbb{T}}
\newcommand{\U}{\mathbb{U}}
\newcommand{\V}{\mathbb{V}}
\newcommand{\W}{\mathbb{W}}
\newcommand{\X}{\mathbb{X}}
\newcommand{\Y}{\mathbb{Y}}
\newcommand{\Z}{\mathbb{Z}}
\newcommand\ca{\mathcal{A}}
\newcommand\cb{\mathcal{B}}
\newcommand\cc{\mathcal{C}}
\newcommand\cd{\mathcal{D}}
\newcommand\ce{\mathcal{E}}
\newcommand\cf{\mathcal{F}}
\newcommand\cg{\mathcal{G}}
\newcommand\ch{\mathcal{H}}
\newcommand\ci{\mathcal{I}}
\newcommand\cj{\mathcal{J}}
\newcommand\ck{\mathcal{K}}
\newcommand\cl{\mathcal{L}}
\newcommand\cm{\mathcal{M}}
\newcommand\cn{\mathcal{N}}
\newcommand\co{\mathcal{O}}
\newcommand\cp{\mathcal{P}}
\newcommand\cq{\mathcal{Q}}
\newcommand\rr{\mathcal{R}}
\newcommand\cs{\mathcal{S}}
\newcommand\ct{\mathcal{T}}
\newcommand\cu{\mathcal{U}}
\newcommand\cv{\mathcal{V}}
\newcommand\cw{\mathcal{W}}
\newcommand\cx{\mathcal{X}}
\newcommand\ocd{\overline{\cd}}

\def\c{\centerline}
\def\ov{\overline}
\def\emp {\emptyset}
\def\pa {\partial}
\def\bl{\setminus}
\def\op{\oplus}
\def\sbt{\subset}
\def\un{\underline}
\def\al {\alpha}
\def\bt {\beta}
\def\de {\delta}
\def\Ga {\Gamma}
\def\ga {\gamma}
\def\lm {\lambda}
\def\Lam {\Lambda}
\def\om {\omega}
\def\Om {\Omega}
\def\sa {\sigma}
\def\vr {\varepsilon}
\def\va {\varphi}

\title{\bf Existence and concentration of ground state solutions for a critical nonlocal Schr\"{o}dinger equation in $\R^2$}

\author{Claudianor O. Alves\thanks{Partially supported by CNPq/Brazil
304036/2013-7, coalves@dme.ufcg.edu.br}\\
{\small  Universidade Federal de Campina Grande} \\ {\small Unidade Acad\^emica de Matem\'{a}tica} \\ {\small CEP: 58429-900, Campina Grande - Pb, Brazil}
\\
\\
Daniele Cassani\thanks{daniele.cassani@uninsubria.it}\vspace{2mm}\\
{\small Dip. di Scienza e Alta Tecnologia, Universit\`{a} degli Studi dell' Insubria}
 \\ {\small via Valleggio 11, 22100 Como, Italy}
\\
\\
Cristina Tarsi\thanks{cristina.tarsi@unimi.it}\vspace{2mm}\\
{\small Dip. di Matematica "F. Enriques", Universit\`{a} degli Studi di Milano}
 \\ {\small via C. Saldini 50, 20133 Milano, Italy}
\\
\\
Minbo Yang\thanks{*Corresponding author, partially supported by NSFC (11571317, 11101374, 11271331) and ZJNSF(LY15A010010), mbyang@zjnu.edu.cn}\vspace{2mm}
\\
{\small  Department of Mathematics, Zhejiang Normal University} \\ {\small  Jinhua, Zhejiang, 321004, P. R. China.}}

\date{}
\maketitle

\begin{abstract}
We study the following singularly perturbed nonlocal Schr\"{o}dinger equation
$$
-\vr^2\Delta u +V(x)u =\vr^{\mu-2}\Big[\frac{1}{|x|^{\mu}}\ast F(u)\Big]f(u)  \quad \mbox{in} \quad \R^2,
$$
where $V(x)$ is a continuous real function on $\R^2$, $F(s)$ is the primitive of $f(s)$, $0<\mu<2$ and $\vr$ is a positive parameter. Assuming that the nonlinearity $f(s)$ has critical exponential growth in the sense of Trudinger-Moser, we establish  the existence and concentration of solutions by variational methods.

 \vspace{0.3cm}

\vspace{0.3cm}

\end{abstract}

\section{Introduction and main results}
The nonlocal elliptic equation
$$
\aligned &-\varepsilon^2\Delta u +V(x)u  =\vr^{\mu-N}\Big[\frac{1}{|x|^{\mu}}\ast F(u)\Big]f(u)  \quad \mbox{in} \quad \R^N,
\endaligned\eqno{(SNS)}
$$
the so-called Choquard equation when $N=3$, appears in the theory of Bose-Einstein condensation and is used to describe the finite-range many-body interactions between particles. Here  $V(x)$ is the external potential, $F(s)$ is the primitive of the nonlinearity $f(s)$ and the parameters $\varepsilon>0$,  $0<\mu<N$.  For $\mu=1$ and $F(s)=\frac12|s|^{2}$, equation $(SNS)$ was investigated by S.I. Pekar in \cite{P1} to study the quantum theory of a polaron at rest. In \cite{Li1} P. Choquard suggested to use it as approximation to Hartree-Fock theory of one-component plasma. This equation was also proposed by R. Penrose in \cite{MPT} as a model for selfgravitating particles and it is known in that context as the Schr\"odinger-Newton equation.

Notice that if $u$ is a solution of the nonlocal equation $(SNS)$ and $x_0\in\R^N$, then the function $v=u(x_0+\varepsilon x)$ satisfies
 $$
-\Delta v +V(x_0+\vr x)v  =\Big[\frac{1}{|x|^{\mu}}\ast F(v)\Big]f(v)  \quad \mbox{in} \quad \R^N.
 $$
This suggests some convergence, as $\vr\to0$, of the family of solutions of $(SNS)$ to a solution $u_0$ of the limit problem
\begin{equation}
-\Delta v +V(x_0)v=\Big[\frac{1}{|x|^{\mu}}\ast F(v)\Big]f(v) \quad \mbox{in} \quad \R^N.
\end{equation}
This is known as semi-classical limit for the nonlocal Choquard equation and we refer for a survey to \cite{AM,BF}. The study of semiclassical states for the Schr\"{o}dinger equation
\begin{equation}\label{S.S}
 -\vr^2\Delta u +V(x)u  =g(u)  \quad \mbox{in} \quad \R^N,
\end{equation}
goes back to the pioneer work \cite{FW} by Floer and Weinstein. Since then, it has been studied
extensively under various hypotheses on the potential and the nonlinearity, see for example
\cite{ABC, DL, DXL, FW, GW, JT, DF2, DF1, R, WX} and
the references therein.
In the study of semiclassical problems for local Schr\"{o}dinger equations, the Lyapunov-Schmidt reduction method has been proved to be one of the most powerful tools. However, this technique  relies on the uniqueness and non-degeneracy of the ground states of the limit problem which is not completely settled for the ground states of the nonlocal Choquard equation
\begin{equation}\label{CC}
-\Delta u +u  =\Big[\frac{1}{|x|^{\mu}}\ast F(u)\Big]f(u) \,\,\ \mbox{in} \,\,\, \mathbb{R}^{N}.
\end{equation}
In \cite{ ML, CCS3, MS1}, have been investigated qualitative properties of solutions and established regularity, positivity, radial symmetry and decaying behavior at infinity.  Moroz and Van
Schaftingen in \cite{MS2} established the
existence of ground states under the assumption of Berestycki-Lions type and for the critical equation in the sense of Hardy-Littlewood-Sobolev inequality.
For $N = 3$, $\mu=1$ and $F(s)=\frac12|s|^2$, by proving the uniqueness and non-degeneracy of the ground states, Wei
and Winter \cite{WW} constructed a family of solutions
 by a Lyapunov-Schmidt type reduction when $\inf V>0$.
 In presence of non-constant electric and magnetic potentials, Cingolani et.al. \cite{CCS} showed that there exists a family of solutions
having multiple concentration regions which are localized by the minima of the potential. Moroz and Van Schaftingen \cite{MS3} used
variational methods and developed a nonlocal penalization technique to show that equation $(SNS)$ has a family of solutions
concentrating at the local minimum of $V$ provided $V$ satisfies some additional assumptions at infinity.
In \cite{YD}, Yang and Ding considered the following equation
 $$
\aligned &-\varepsilon^2\Delta u +V(x)u  =\Big[\frac{1}{|x|^{\mu}}\ast u^p\Big]u^{p-1}, \quad \mbox{in} \quad \R^3.
\endaligned
$$
and by using variational methods, they were able to obtain the existence of solutions which vanish at infinity for suitable parameters $p, \mu$. In \cite{AYang}, Alves and Yang proved the existence, multiplicity and concentration of solutions for the same equation by penalization methods and Lusternik-Schnirelmann theory.

\par
Let us recall the following form of the Hardy-Littlewood-Sobolev inequality, see \cite{LL}, which will be frequently used throughout this paper:

\begin{Prop}[Hardy-Littlewood-Sobolev\ inequality]\label{HLS}
Let $s, r>1$ and $0<\mu<N$ with $1/s+\mu/N+1/r=2$. Let $f\in
L^s(\R^N)$ and $h\in L^r(\R^N)$. There exists a sharp constant
$C(s,N,\mu,r)$, independent of $f,h$, such that
$$
\int_{\R^N}[\frac{1}{|x|^{\mu}}\ast f(x)]h(x)\leq
C(s,N,\mu,r) |f|_s|h|_r.
$$
  \end{Prop}

 By the Hardy-Littlewood-Sobolev inequality,
$$
\int_{\R^2}\Big[\frac{1}{|x|^{\mu}}\ast F(u)\Big]F(u)
$$
is well defined if $F(u)\in L^s(\R^N)$ for $s>1$ given by
$$
\frac2s+\frac{\mu}{N}=2.
$$
This means we must require $$F(u)\in L^{\frac{2N}{2N-\mu}}(\R^N).$$
In order to preserve the variational structure of the problem in $\R^N$, $N\geq 3$ for the prototype model $F(u)=|u|^p$, we must require by means of Sobolev's embedding that the exponent $p$ satisfies
$$
\frac{2N-\mu}{N}\leq p\leq \frac{2N-\mu}{N-2}.
$$
The confining exponents above play the role of critical exponents for the nonlocal Choquard equation in $\R^N$, $N\geq 3$. Most of the works afore mentioned are set in $\R^N$, $N\geq 3$, with non-critical growth nonlinearities and to the authors best knowledge no results are available on the existence and concentration of solutions for the nonlocal Choquard equation with upper-critical exponent $\frac{2N-\mu}{N-2}$ but only in the case of the lower-critical exponent $\frac{2N-\mu}{N}$, see \cite{MS4}.

 The case $N=2$ is very special, as for bounded domains $\Omega\subset\R^2$ the corresponding Sobolev embedding yields $H^1_0(\Omega)\subset L^{q}(\Omega)$ for all $q\geq1$, but $H^1_0(\Omega)\nsubseteqq L^{\infty}(\Omega)$. In dimension $N=2$, the Pohozaev-Trudinger-Moser inequality \cite{P,M} can be seen as a substitute of the Sobolev inequality as it establishes the following sharp maximal exponential integrability for functions with membership in $H^1_0(\Omega)$:
\[
\sup_{u\in H_0^1(\Omega) \; : \; \|\nabla u\|_2\leq
1}\int_{\Omega} e^{\alpha
 u^2}\leq C|\Omega| \quad\text{if }\alpha \leq 4\pi,
\]
for a positive constant which depends only on $\alpha$ and where $|\Omega|$ denotes Lebesgue measure of $\Omega$. As a consequence we say that a function $f(s)$ has
\emph{critical exponential growth} if there exists $\alpha_0>0$ such that
\begin{equation}\label{ecg}
\lim_{|s|\to +\infty}\frac{|f(s)|}{e^{\alpha s^2}}=0,\ \ \forall \alpha>\alpha_0,\ \ \hbox{and} \ \ \lim_{|s|\to +\infty}\frac{|f(s)|}{e^{\alpha s^2}}=+\infty,\ \ \forall \alpha<\alpha_0.
\end{equation}
This notion of criticality was introduced by Adimurthi and Yadava \cite{AY}, see also de Figueiredo, Miyagaki and Ruf \cite{DMR}.
The first version of the Pohozaev-Trundiger-Moser inequality in $\R^2$ was established by Cao in \cite{Cao}, see also \cite{O, AT,CST}, and reads as follows
\begin{lem} \label{Trudinger-Moser}
If $\alpha >0$ and $u\in H^{1}(\mathbb{R}^2)$, then
\begin{equation}\label{TM1}
\int_{\mathbb{R}^2}\Big[ e^{ \alpha | u | ^2}
-1 \Big] <\infty .
\end{equation}
Moreover, if $|\nabla u| _2^2\leq 1$,
 $ | u |_2\leq M<\infty $, and $\alpha <\alpha
_0=4\pi$, then there exists a constant
$C$, which depends only on $M$ and $\alpha $,
such that
\begin{equation}\label{TM2}
\int_{\mathbb{R}^2}\big[ e^{\alpha | u| ^{2}}
-1 \big] \leq C(M,\alpha ).
\end{equation}
\end{lem}
\noindent
We refer the reader to \cite{AY,YR} for related problems and \cite{CST, LL1, YY} for recent advances on this topic.  Actually just a few papers deal with semiclassical states for local Schr\"{o}dinger equations with critical exponential growth. In \cite{OSo}, do \'O and Souto proved the existence of solutions concentrating around local minima of of $V(x)$ which are not necessarily nondegenerate. For $N$-Laplacian equation in $\R^N$, Alves and Figueiredo \cite{AF} studied the multiplicity of semiclassical solutions with Rabinowitz type assumption on the potential. Recently, do \'O and Severo \cite{OSe} and do \'O, Moameni and Severo \cite{OMU}  also studied a class of quasilinear Schr\"{o}dinger equations in $\R^2$ with critical exponential growth.

Hence it is quite natural to wonder if the existence and concentration results for local Schr\"{o}dinger equations still hold for the nonlocal equation with critical growth in the sense of Pohozaev-Trudinger-Moser. The purpose of this paper is two-fold: on the one hand we study the existence of nontrivial solution for the critical nonlocal equation with periodic potential, namely we consider the equation
\begin{equation}\label{A1}
-\Delta u + W(x)u  =\Big(\frac{1}{|x|^{\mu}}\ast F(u)\Big)f(u), \,\,\ \mbox{in} \,\,\, \mathbb{R}^{2}.
\end{equation}
and assume  for the potential the following
\begin{itemize}
\item[$(W_1)$] $W(x)\geq W_0>0$ in $\R^2$ for some $W_0>0$;
\item[$(W_2)$] $W(x)$ is a  1-periodic continuous function.
\end{itemize}
and for the nonlinearity $f$ which satisfies the following
\begin{itemize}
\item[$(f_1)$](i)$f(s)=0 \quad \forall s \leq 0$, $\displaystyle 0\leq f(s) \leq C e^{4\pi s^2},
\quad s\geq 0$;
\newline (ii) $\exists\, s_0>0,  M_0>0, \ \hbox{and} \ q\in (0,1]$ such that
 $0<s^q F(s)\leq M_0f(s),\:\forall\, |s|\geq s_0$.
\item[$(f_2)$] There exists $\displaystyle p>
\frac{2-\mu}{2} $ and $ C_p>0$ such that $f(s)\sim C_p s^p$, as $s\to
0$.


\item[$(f_3)$]  There exists $K>1$ such that $f(s)s>KF(s)$ for
all $s>0$, where $F(t)=\int^t_0f(s)ds$.

\item[$(f_4)$] $\displaystyle
    \lim_{s\to +\infty} \frac{sf(s)F(s)}{e^{8\pi s^2}}\geq\beta,  $
 with  $\beta>\displaystyle \inf_{\rho>0}
\frac{e^{\frac{4-\mu}4V_0\rho^2}}{16 \pi^2
\rho^{4-\mu}}\frac{(4-\mu)^2}{(2-\mu)(3-\mu) }$.

\end{itemize}
Our first main result reads as follows
\begin{thm}\label{thm-Existence}
  Assume $0<\mu<2$, suppose that the potential $V$ satisfies
$(W_1)-(W_2)$ and the nonlinearity $f$ satisfies conditions
$(f_1)-(f_4)$. Then equation \eqref{A1} has a ground state solution in
$H^1(\mathbb R^2)$.
\end{thm}
On the other hand, we establish existence and concentration of semiclassical ground state solutions of the following equation
\begin{equation}\label{EC}
\aligned &-\varepsilon^2\Delta u +V(x)u =\vr^{\mu-2}\Big[\frac{1}{|x|^{\mu}}\ast  F(u)\Big]f(u) \,\,\, \mbox{in} \,\,\, \R^2.
\endaligned
\end{equation}
Here we assume the following conditions on $V$:
\begin{itemize}
\item[$(V_1)$] $V(x)\geq V_0>0$ in $\R^2$ for some $V_0>0$;

\item[$(V_2)$]
$
0<\inf_{x\in \R^2}V(x)=V_0<V_{\infty}=\liminf_{|x|\to \infty}V(x)<\infty.
$
\end{itemize}
The condition $(V_2)$ was introduced by Rabinowitz in \cite{R}. Hereafter, we will denote by
$$
{M}=\{x\in\R^2:V(x)=V_0\},
$$
the minimum points set of $V(x)$.

\noindent We also assume that  that the nonlinearity enjoys the following
\begin{itemize}
\item[$(f_5)$] $\displaystyle s\to f(s)$ \quad is strictly increasing on  $(0, +\infty)$.
\end{itemize}
Then we prove our second main result
\begin{thm}\label{T1}
Suppose that the nonlinearity $f(s)$ satisfies $(f_1)-(f_5)$ and the potential function $V(x)$ satisfies assumptions $(V_1)-(V_2)$. Then, for any $\vr>0$ small, problem $(\ref{EC})$ has at least one positive ground state solution. Moreover, let $u_\vr$ denotes one of these positive solutions with $\eta_\vr\in\R^2$ its global maximum, then
$$
\lim_{\varepsilon\to 0}V(\eta_\varepsilon)=V_0.
$$
\end{thm}

\noindent \underline{Notation}: \\
\noindent $\bullet$ $C$, $C_i$ denote positive constants.\\
\noindent $\bullet$ $B_R$ denote the open ball centered at the origin with
radius $R>0$. \\
\noindent $\bullet$ $C_0^{\infty}(\R^2)$ denotes  the space of the functions
infinitely differentiable with compact support in $\R^2$. \\
\noindent $\bullet$ For a mensurable function $u$, we denote by $u^{+}$ and $u^{-}$ its positive and negative parts respectively, given by
$$
u^{+}(x)=\max\{u(x),0\} \quad \mbox{and} \quad u^{-}(x)=\min\{u(x),0\}.
$$
\noindent $\bullet$ In what follows, we denote by $\|\,\,\,\,\|$ and $|\,\,\,\,|_s$ the usual norms of the spaces $H^{1}(\mathbb{R}^{2})$ and $L^{s}(\mathbb{R}^{2})$ respectively. \\
\noindent $\bullet$ \, Let $E$ be a real Hilbert space and $I:E \to \R$ a functional of class $\mathcal{C}^1$.
 We say that $\{u_n\}\subset E$ is a  Palais-Smale ($(PS)$ for short) sequence at $c$ for $I$ if $\{u_n\}$ satisfies
$$
I(u_n)\to c \,\,\, \mbox{and} \,\,\,\, I'(u_n)\to0,  \,\,\, \mbox{as} \,\,\, n\to\infty.
$$
Moreover, $I$ satisfies the $(PS)$ condition at level $c$, if any $(PS)$ sequence $\{u_n\}$ such that $I(u_n)\to c$ possesses a convergent subsequence.

\section{A critical nonlocal equation with periodic potential: proof of Theorem \ref{thm-Existence}}

In \cite{AY2}, Alves and Yang  studied equation \eqref{A1} under hypothesis $(W1)$ and $(W_2)$ for the potential and the following conditions on the nonlinearity $f : \R^+\to \R$ of class $\mathcal{C}^1$:
$$
f(0)=0,\ \ \ \lim_{s\to 0}{f'(s)}=0. \eqno{(f'_1)}
$$
It is of critical growth at infinity with $\alpha_0=4\pi$. Moreover, there exists $C_0$ such that
$$
| f'(s)|  \leq \,\,\,  C_0e^{4 \pi s^{2}},  \,\,\,\, \forall s> 0. \eqno{(f'_2)}
$$
There exists ${\theta}>2$ such that
$$
0<\theta F(s)\leq 2f(s)s,\ \  \forall s>0, \eqno{(f'_3)}
$$
Furthermore, they suppose that there exists $p>\frac{4-\mu}{2}$, such that
$$
F(s)\geq C_ps^p,\,\,\,\, \forall s>0 \eqno{(f'_4)}
$$
where
$$
C_p>\frac{[\frac{4\theta(p-1)}{(2-\mu)(\theta-2)}]^{\frac{p-1}{2}}S^{p}_p}{p^{\frac{p}{2}}}.
$$
and
$$
\displaystyle S_p=\displaystyle\inf_{u\in H^{1}(\R^2), u\neq0}\frac{\displaystyle\left(\int_{\R^2}\big(|\nabla u|^2+|W|_\infty|u|^2\big)\right)^{1/2}}{\displaystyle\left(\int_{\R^2}\Big[ \frac{1}{|x|^{\mu}}\ast |u|^p\Big]|u|^p\right)^{\frac{1}{2p}}}.
$$

\noindent Combining the above estimates with the Hardy-Littlewood-Sobolev inequality and some results due to P.L. Lions, the following existence result was obtained in \cite{AY2}.
\begin{thm}\label{AQ1}
Suppose that conditions $(f'_1)-(f'_4)$ hold. Then problem \eqref{A1} has at least one ground state solution $w$.
\end{thm}

\noindent A key tool in \cite{AY2} is assumption $(f'_4)$ which enables one to obtain estimates of the Mountain-Pass level for the energy functional related to the nonlocal Choquard equation, for $0<\mu<2$,
\begin{equation}\label{A}
\left\{
\begin{array}{l}
\displaystyle -\Delta u +W(x) u  =\Big(\frac{1}{|x|^{\mu}}\ast F(u)\Big)f(u), \,\,\ \mbox{in} \,\,\, \mathbb{R}^{2}, \\
u \in H^{1}(\mathbb{R}^{2})\\ u(x)>0 \hbox{ for all } x\in \R^2.
\end{array}
\right.
\end{equation}
Condition $(f_4')$ involves the explicit value of the best constant of the embedding $H^1\hookrightarrow L^p$, $p\in (2,\infty)$, which is so far unknown and still an open challenging problem. In terms of the nonlinear source, condition $(f_4')$ prescribe a global growth which can not be actually verified. This somehow affects possible further applications.  The aim of this section is to overcome condition $(f_4')$ which we replace with the assumption $(f_4)$. For this purpose, we set
$$
W_{\rho}:=\sup_{|x|\leq \rho}W(x)
$$
and
$$
\mathcal W:=\inf_{\rho>0}
\frac{e^{\frac{4-\mu}4W_{\rho}\rho^2}}{16 \pi^2
\rho^{4-\mu}}\frac{(4-\mu)^2}{(2-\mu)(3-\mu) }.
$$
Notice that if $W(x)$ is continuous and $(W_2)$ is satisfied, then
$W_{\rho}$ is a positive continuous function and $\mathcal W$ can
be attained by some $\rho>0$. Moreover, it is worth to point out  that assumption $(f_1)-(ii)$ implies that for any $\eta>0$
there exists $C_{\eta}>0$ and $s_{\eta}$ such that for all $s\geq s_{\eta}$
\begin{equation}\label{ARcond}
  \eta f(s)\geq F(s)
\end{equation}
and as $s$ is large enough
$$
F(s)\geq C_{\eta}e^{s^{q+1}}.
$$
On the other hand,
$(f_1)-(ii)$ implies for some $\gamma >0$
$$
F(s)\leq e^{\gamma s^{2}}-1, \qquad \hbox{for any } s>0
$$
which agrees with $(f_2)$. Notice also that assumptions $(f_2)$ and $(f_3)$ yield
$$
K>\frac{4-\mu}{2}>1.
$$
 Assumption $(f_4)$ is inspired by
\cite{DMR,YY}, but here we have the extra difficulty to handle integrals where both the two nonlinearities $F(s)$ and
$sf(s)$ appear simultaneously. This situation forces us to assume
condition $(f_4)$ which is sharper than the following assumption of \cite{DMR}
\begin{equation}\label{f_4 bis}
 \lim_{s\to +\infty} \frac{F(s)}{e^{4\pi s^2}}\geq\gamma.
\end{equation}
 Actually, condition \eqref{f_4 bis}, combined with \eqref{ARcond} implies
$$
\lim_{s\to +\infty} \frac{sf(s)}{e^{4\pi s^2}}\geq\gamma
\eta^{-1} \quad \hbox{for any } \eta >0,
$$
so that $(f_4)$ is trivially satisfied for any choice of
$\gamma>0$. Finally, note that $(f_4)$ together with \eqref{ARcond}
still imply
$$
\lim_{s\to +\infty} \frac{sf(s)}{e^{4\pi s^2}}=+\infty,
$$
but it may happen that
$$
\lim_{s\to +\infty} \frac{F(s)}{e^{4\pi s^2}}=0
$$
 in contrast with \eqref{f_4 bis}. This is the case, for instance, if
$$
F(s)\sim \frac{e^{4\pi s^2}}s \text{ and } f(s)\sim 8\pi e^{4\pi
s^2}, \quad s \to +\infty.
$$
Since we are looking for positive solutions $u\geq 0$, from
now on we assume $f(s)=0$ for $s\leq 0$.
The energy functional associated with problem \eqref{A} is given
by
$$
\aligned \Phi_{W}(u)=\frac12\|u\|^2_{W}-\mathfrak{F}(u),
\endaligned
$$
where
$$
\mathfrak{F}(u)=\frac12\int_{\R^2}\Big[\frac{1}{|x|^{\mu}}\ast
F(u)\Big]F(u)
$$
and
 $$
 \|u\|_{W}:=\left(\int_{\R^2}|\nabla u|^2+ W(x)|u|^2\right)^{1/2}
 $$
Let $E$ denote the space $ H^1(\mathbb R^2)$ equipped with the
norm $\|u\|_{W}$, which is equivalent to the standard Sobolev
norm.

\noindent As a
 consequence of Cao's inequality in Lemma \ref{Trudinger-Moser}, $(f_2)$ and H\"{o}lder's inequality
 we have $F(u)\in L^{\frac{4}{4-\mu}}(\mathbb R^2)$ (note that $(f_2)$ is weaker then
 $(f_1')$ of \cite{AY2}), and the functional $\Phi_W(u)$ is $\mathcal C^1(E)$ thanks
to a generalization of a Lions' result recently proved in
\cite{doOdSdMS}. Then the Mountain Pass geometry can be proved as in \cite{AY2}.
 By the Ekeland Variational Principle \cite{}, there exists a (PS) sequence $(u_n) \subset
E\subset H^1(\mathbb R^2)$ such that
$$
\Phi'_{W}(u_n)\rightarrow0,\quad \Phi_{W}(u_n)\rightarrow {m_{W}},
$$
where the Mountain Pass e $m_{W}$ can be characterized by
\begin{equation} \label{m}
0<m_{W}:=\inf_{\gamma \in \Gamma} \max_{t\in [0,1]}
\Phi_{W}(\gamma (t))
\end{equation}
with
$$
\Gamma:=\left\{\gamma \in \mathcal C^1([0,1], E): \gamma(0)=0,
\Phi_{W}(\gamma(1))<0\right\}.
$$
\begin{lem}\label{MPlevel-estimate}
The mountain pass level $m_{W}$ satisfies
$$
m_{W}<\frac{4-\mu}{8}.
$$
\end{lem}
\begin{proof}
It is enough  to prove that there exists s a function $w\in E,
\|w\|_{W}=1$, such that
$$
\max_{t\geq 0}\Phi_{W}(tw)<\frac{4-\mu}{8}.
$$
Let us introduce the following Moser type
functions supported in $B_{\rho}$ by
$$
\overline{w}_n=\frac{1}{\sqrt{2\pi}}\left\{%
\begin{array}{ll}
\displaystyle    \sqrt{\log n}, & 0\leq |x|\leq \frac{\rho}n, \\
\\
\displaystyle    \frac{\log(\rho/|x|)}{\sqrt{\log n}}, & \frac{\rho}n\leq |x|\leq {\rho}, \\
\\
0,& |x|\geq {\rho}. \\
\end{array}%
\right.
$$
One has that
\begin{eqnarray}\label{delta_n1}
\nonumber \|\overline w_n\|^2_{W}&=&\int_{B_{\rho}}|\nabla\overline
w_n|^2+\int_{B_{\rho}}W(x)|\overline w_n|^2\\
\nonumber &\leq&
\int_{{\rho}/n}^{\rho}\frac{dr}{r\log
n}dr+W_{\rho}\int_0^{{\rho}/n}\log n\,
rdr+W_{\rho}\int_{{\rho}/n}^{\rho}\frac{\log^2(\rho/r)}{\log n}\,rdr\\
\nonumber&=&1 +\delta_n,
\end{eqnarray}
where
\begin{eqnarray}\label{delta_n2}
\delta_n= W_{\rho}\rho^2\left[\frac{1}{4\log
n}-\frac{1}{4n^2\log n}-\frac{1}{2n^2}\right]>0.
\end{eqnarray}
And then, setting $w_n=\overline w_n/\sqrt{1+\delta_n}$, we get
$\|w_n\|_{W}=1$.

 \noindent We claim that there exists $n$ such that
\begin{equation}\label{claim}
\max_{t\geq 0}\Phi_{W}(tw_n)<\frac{4-\mu}{8}.
\end{equation}
Let us argue by contradiction and suppose this is not the case, so that for all
$n$ let $t_n>0$ be such that
\begin{equation}\label{bycontr-assump}
\max_{t\geq 0}\Phi_{W}(tw_n)=\Phi_{W}(t_nw_n)\geq\frac{4-\mu}{8},
\end{equation}
then $t_n$ satisfies $\frac{d}{dt}\Phi_{W}(tw_n)|_{t=t_n}=0$, then
\begin{equation}\label{t_n^2=}
t^2_n= \int_{\mathbb R^2}\left[\frac{1}{|x|^{\mu}}\ast
F(t_nw_n)\right]t_nw_nf(t_nw_n),
\end{equation}
it follows from \eqref{bycontr-assump} that
\begin{equation}\label{est-t_n^2}
t_n^2\geq \frac{4-\mu}{4}.
\end{equation}
Let us estimate from below the quantity $t_n^2$. Taking advantage of equation
\eqref{t_n^2=}, thanks to $(f_4)$ we have for any $\varepsilon>0$,
\begin{equation}\label{estimate-sfF}
  sf(s)F(s)\geq (\beta-\varepsilon)e^{8\pi s^2} \quad \hbox{ for all
}s\geq s_{\varepsilon}
\end{equation}
and thus
\begin{eqnarray*}
 t_n^2&\geq& \int_{B_{\rho/n}}t_nw_nf(t_nw_n)dy\int_{B_{\rho/n}}\frac{1}{|x-y|^{\mu}}F(t_nw_n)\,dx\\
&=& \int_{B_{\rho/n}}t_n\frac{\sqrt{\log n}}{\sqrt{
2\pi}}f\left(t_n\frac{\sqrt{\log n}}{\sqrt{
2\pi}}\right)dy\int_{B_{\rho/n}}\frac{1}{|x-y|^{\mu}}
F\left(t_n\frac{\sqrt{\log n}}{\sqrt{2\pi}}\right)\,dx\\
 &\geq& (\beta - \varepsilon)e^{4 t_n^2 (1+\delta_n)^{-1} \log
  n}\int_{B_{\rho/n}}dy\int_{B_{\rho/n}}\frac{1}{|x-y|^{\mu}}\,dx.
\end{eqnarray*}
Notice that $B_{\rho/n-|x|}(0)\subset B_{\rho/n}(x)$ since $|x|\leq \rho/n$, the last integral can be estimated as follows
\begin{eqnarray} \label{intconv}
\aligned
  \int_{B_{\rho/n}}dy\int_{B_{\rho/n}}\frac{dx}{|x-y|^{\mu}}
 &=\int_{B_{\rho/n}}dx\int_{B_{\rho/n}(x)}\frac{dz}{|z|^{\mu}}\\
  &\geq
  \int_{B_{\rho/n}}dx\int_{B_{\rho/n-|x|}}\frac{dz}{|z|^{\mu}}\\
  &=\frac{2\pi}{2-\mu}\int_{B_{\rho/n}}\left(\frac{\rho}n-|x|\right)^{2-\mu}\\
  &=\frac{4\pi^2}{2-\mu}\int_0^{\rho/n}\left(\frac{\rho}n-r\right)^{2-\mu}rdr\\
  &=\frac{4\pi^2}{(2-\mu)(3-\mu)(4-\mu)}\left(\frac{\rho}{n}\right)^{4-\mu}\\
 &=C_{\mu}\left(\frac{\rho}{n}\right)^{4-\mu},
 \endaligned
\end{eqnarray}
where
\begin{eqnarray*}
C_{\mu}=\frac{4\pi^2}{(2-\mu)(3-\mu)(4-\mu)}.
\end{eqnarray*}
Consequently, we obtain
\begin{eqnarray*}
 t_n^2&\geq&\frac{4\pi^2 (\beta - \varepsilon)}{(2-\mu)(3-\mu)(4-\mu)} e^{4 t_n^2 (1+\delta_n)^{-1}\log
  n}\left(\frac{\rho}{n}\right)^{4-\mu}\\
&=&\frac{4\pi^2 (\beta - \varepsilon)
\rho^{4-\mu}}{(2-\mu)(3-\mu)(4-\mu)}e^{\log n[4
(1+\delta_n)^{-1}t_n^2-(4-\mu)]}
\end{eqnarray*}
which, recalling \eqref{est-t_n^2}, means that $t_n$ is bounded and yields
$$
t_n^2\longrightarrow \left(\frac{4-\mu}{4}\right)^+
$$
as $n$ goes to infinity. Moreover, as a byproduct we also have that for some $C>0$
\begin{eqnarray*}
  \log n[4
(1+\delta_n)^{-1}t_n^2-(4-\mu)]\leq C,
\end{eqnarray*}
that is
\begin{equation}\label{est-t_n^2-bis}
\frac{t^2_n}{1+\delta_n}=
\frac{4-\mu}{4}+{\textrm{O}}\left(\frac{1}{\log n}\right).
\end{equation}
This estimate will be used to obtain a finer estimate than \eqref{est-t_n^2}. Notice first that by
$(f_1)$ and $(f_2)$ we have
\begin{equation}\label{estimate-F}
  F(s)\leq Cs{^\frac{4-\mu}{2}}+Mf(s)\leq Cs{^\frac{4-\mu}{2}}+C(e^{4\pi s^2}-1).
\end{equation}
Next define
$$
A_n=\left\{y\in B_{\rho}: t_nw_n(y)>s_{\varepsilon}\right\}\quad \hbox{and} \quad
B_n= B_{\rho}\setminus A_n,
$$
where $s_{\varepsilon}$ was introduced in \eqref{estimate-sfF}.
By \eqref{estimate-sfF} we know
\begin{multline*}
\aligned
t^2_n&=\int_{\mathbb R^2}\left(\frac{1}{|x|^{\mu}}\ast
F(t_nw_n)\right)t_nw_nf(t_nw_n)\,dy\\
&=\int_{B_{\rho}}\left(\frac{1}{|x|^{\mu}}\ast
F(t_nw_n)\right)t_nw_nf(t_nw_n)\,dy\\
&=\int_{A_n}\left(\frac{1}{|x|^{\mu}}\ast
F(t_nw_n)\right)t_nw_nf(t_nw_n)\,dy+\int_{B_n}\left(\frac{1}{|x|^{\mu}}\ast
F(t_nw_n)\right)t_nw_nf(t_nw_n).
\endaligned
\end{multline*}
Combining Hardy-Littlewood-Sobolev inequality with
\eqref{estimate-F} one has
\begin{multline}\label{finer-est1}
\int_{B_n}\left(\frac{1}{|x|^{\mu}}\ast
F(t_nw_n)\right)t_nw_nf(t_nw_n)\leq C
\|F(t_nw_n)\|_{\frac{4}{4-\mu}}\|\chi_{B_n}t_nw_nf(t_nw_n)\|_{\frac{4}{4-\mu}}\\
\leq\left[ C \|t_nw_n\|_{2}+C\left\{\int_{\mathbb
R^2}e^{4\pi\frac{4}{4-\mu}
t_n^2w_n^2}-1\right\}^{\frac{4-\mu}{4}}\right]\|\chi_{B_n}t_nw_nf(t_nw_n)\|_{\frac{4}{4-\mu}}.
\end{multline}
By \eqref{est-t_n^2-bis}, since $\|\nabla
\overline w_n\|_2=1$ and $\overline w^2_n\leq 2\pi \log n $, we obtain
\begin{multline*}
  \int_{\mathbb
R^2}e^{4\pi\frac{4}{4-\mu}
t_n^2w_n^2}-1\leq\int_{B_\rho}e^{4\pi\frac{4}{4-\mu}
t_n^2w_n^2}\leq \int_{B_\rho}e^{4\pi(1+\frac C{\log n})\overline
w_n^2}\leq\int_{B_\rho}Ce^{4\pi\overline w_n^2}\leq C,
\end{multline*}
due to the Pohozaev-Trudinger-Moser inequality. Since $t_nw_n\to 0$
a.e. and $t_nw_n$ is bounded on $B_n$, applying
the Lebesgue dominated convergence theorem, we obtain

$$\|\chi_{B_n}t_nw_nf(t_nw_n)\|_{\frac{4}{4-\mu}}\to 0,$$ as $n\to \infty$.
Consequently,
\begin{equation}\label{est-t_n^2-tris}
t_n^2=\int_{A_n}\left(\frac{1}{|x|^{\mu}}\ast
F(t_nw_n)\right)t_nw_nf(t_nw_n)\,dy+\rm{o}(1),
\end{equation}
where o$(1)$ is actually positive.

Buying the same lines we can estimate the convolution term as follows
\begin{eqnarray*}
  t_n^2&\geq&  \int_{A_n}t_nw_nf(t_nw_n)\,dy\int_{A_n}\frac{F(t_nw_n)}{|x-y|^{\mu}}\,dx+%
\int_{A_n}t_nw_nf(t_nw_n)\,dy\int_{B_n}\frac{F(t_nw_n)}{|x-y|^{\mu}}\,dx\\
&\geq&
\int_{A_n}t_nw_nf(t_nw_n)\,dy\int_{A_n}\frac{F(t_nw_n)}{|x-y|^{\mu}}\,dx+\rm{o}(1).
\end{eqnarray*}
 By the definition of $w_n$, we observe that

\begin{equation}\label{calculA_n}
A_n=\{0<|x|<\rho
e^{-s_{\varepsilon}\sqrt{2\pi(1+\delta_n)}\sqrt{\log n}}\}\supset
B_{\frac{\rho}{n}},
\end{equation}
then
\begin{eqnarray}\label{intAnAn}
\nonumber  t_n^2&\geq&  \int_{A_n}t_nw_nf(t_nw_n)\,dy\int_{A_n}\frac{F(t_nw_n)}{|x-y|^{\mu}}\,dx \\
\nonumber&\geq&
\int_{B_{\rho/n}}t_nw_nf(t_nw_n)dy\int_{B_{\rho/n}}\frac{F(t_nw_n)}{|x-y|^{\mu}}
\,dx\\
\nonumber&&\quad+\int_{\frac{\rho}{n}\leq |x|\cap x\in
A_n}t_nw_nf(t_nw_n)dy\int_{B_{\rho/n}}\frac{F(t_nw_n)}{|x-y|^{\mu}}
\,dx   \\
\nonumber &&\quad+\int_{B_{\rho/n}}t_nw_nf(t_nw_n)dy\int_{\frac{\rho}{n}\leq
|x|\cap x\in A_n}\frac{F(t_nw_n)}{|x-y|^{\mu}} \,dx\\
\nonumber &&\quad+\int_{\frac{\rho}{n}\leq |x|\cap x\in
A_n}t_nw_nf(t_nw_n)dy\int_{\frac{\rho}{n}\leq |x|\cap x\in
A_n}\frac{F(t_nw_n)}{|x-y|^{\mu}} \,dx
\\
\nonumber&:=&I_1+I_2+I_3+I_4\\
&\geq& I_1\geq(\beta -\varepsilon)e^{8\pi t_n^2
w_n^2}\int_{B_{\rho/n}}dy\int_{B_{\rho/n}}\frac{1}{|x-y|^{\mu}}
\,dx
\end{eqnarray}
where we have used the fact that $w_n$ is constant on the ball $B_{\rho/n}$. Thanks to
\eqref{intconv} we have
\begin{eqnarray}\label{intAnAn2}
\nonumber I_1&\geq& (\beta -\varepsilon)e^{4 t_n^2(1+\delta_n)^{-1} \log n} \int_{|y|\leq
\frac{\rho}{n}}dy\int_{|x|\leq
\frac{\rho}{n}}\frac{1}{|x-y|^{\mu}}\,dx\\
&\geq& (\beta -\varepsilon)C_\mu  e^{4
t_n^2(1+\delta_n)^{-1} \log n}\left(\frac{\rho}{n}\right)^{4-\mu}
\end{eqnarray}
and hence, recalling the definition of $\delta_n$ in \eqref{delta_n2}, we also have
\begin{eqnarray*}
I_1&\geq& (\beta -\varepsilon)C_{\mu}\rho^{4-\mu}e^{4 t_n^2(1+\delta_n)^{-1} \log
n-(4-\mu)\log n}\\
&\geq& (\beta -\varepsilon)C_{\mu}\rho^{4-\mu}e^{(4-\mu)\log n[
(1+\delta_n)^{-1} -1]}\\
&\geq& (\beta -\varepsilon)C_{\mu}\rho^{4-\mu}e^{-(4-\mu)\delta_n\log
n}\\
&=&(\beta -\varepsilon)C_{\mu}\rho^{4-\mu}e^{-(4-\mu)W_{\rho}\rho^2\left[\frac{1}{4}-\frac{1}{4n^2}-\frac{\log
n}{2n^2}\right]}\\
&\to&
(\beta -\varepsilon)C_{\mu}\rho^{4-\mu}e^{-\frac{4-\mu}4W_{\rho}\rho^2},
\end{eqnarray*}
as $n\to +\infty$. Combining the previous inequality with \eqref{intAnAn} and passing
to the limit we get
$$
\frac{4-\mu}{4}\geq (\beta -\varepsilon)
C_{\mu}\rho^{4-\mu}e^{-\frac{4-\mu}4W_{\rho}\rho^2}
$$
and since $\varepsilon$ is arbitrary, in turn
$$
\beta \leq \frac{4-\mu}{4C_{\mu}
\rho^{4-\mu}}e^{\frac{4-\mu}4W_{\rho}\rho^2}=
\frac{e^{\frac{4-\mu}4W_{\rho}\rho^2}}{16 \pi^2
\rho^{4-\mu}}\frac{(4-\mu)^2}{(2-\mu)(3-\mu) }
$$
However, by definition of $\mathcal W$ and since $\beta>\mathcal{W}$ by
 $(f_4)$, there exists $\rho>0$ such that
\begin{equation}\label{beta_below}
\beta>\frac{e^{\frac{4-\mu}4W_{\rho}\rho^2}}{16 \pi^2
\rho^{4-\mu}}\frac{(4-\mu)^2}{(2-\mu)(3-\mu) }
\end{equation}
and thus a contradiction and this concludes the proof.
\end{proof}
\begin{Rem}
It is worth to mention that actually estimate \eqref{intAnAn} can be improved, in the sense that the constant $\mathcal W$ can be sharpened by exploiting $I_2$,
$I_3$ and $I_4$ and some additional technical growth assumptions on $f(s)$, which we omit here since do not bring to effective advantages in this context.
\end{Rem}

In the spirit of \cite{YY} we next prove that the limit of a Palais-Smale
sequence for $\Phi_V$ yields a weak solution to \eqref{A}.
\begin{lem}\label{lem-PS}
Assume $(W_1)-(W_2), (f_1)-(f_4)$ and let
$\{u_n\}\subset E$ be a Palais-Smale sequence for
$\Phi_W$, i.e.
$$
\Phi_W(u_n)\to c\quad\text{ and }\quad  \Phi_W'(u_n)\to 0 \quad \hbox{in }
E^*, \quad \hbox{ as } n\to +\infty
$$
Then there exists $u\in E$ such that, up to subsequence,
$u_n\rightharpoonup u$ weakly in $E$,
\begin{equation}\label{convFF}
\left[\frac{1}{|x|^\mu}\ast
F(u_n)\right]F(u_n)\rightarrow\left[\frac{1}{|x|^\mu}\ast
F(u)\right]F(u), \quad \hbox{ in }\quad  L^1_{loc}(\mathbb R^2)
\end{equation}
and $u$ is a weak solution of \eqref{A}.
\end{lem}
\begin{proof}By hypothesis we have
\begin{equation}\label{convPhi}
\frac 12 \|u_n\|_W^2-\frac 12 \int_{\mathbb
R^2}\left[\frac{1}{|x|^\mu}\ast F(u_n)\right]F(u_n)\to c
\end{equation}
as well as
\begin{eqnarray}\label{convPhi'}
\nonumber\left|\int_{\mathbb R^2}\nabla u_n\nabla v+Wu_nv-
\int_{\mathbb R^2}\left[\frac{1}{|x|^\mu}\ast F(u_n)\right]f(u_n)v
\right|\leq \tau_n \|v\|_W
\end{eqnarray}
for all $v\in E$, where $\tau_n\to 0$ as $n\to
+\infty$. Taking $v=u_n$ in \eqref{convPhi'} we obtain
\begin{equation}\label{convPhi''}
\left|\|u_n\|_W^2- \int_{\mathbb R^2}\left[\frac{1}{|x|^\mu}\ast
F(u_n)\right]u_nf(u_n) \right| \leq \tau_n \|u_n\|_W.
\end{equation}
By
$(f_1)$ that for any $s>0$ one has $sf(s)\geq K F(s)$ . Then,
\begin{eqnarray*}
  \int_{\mathbb
R^2}\left[\frac{1}{|x|^\mu}\ast F(u_n)\right]u_nf(u_n)\geq K
\int_{\mathbb R^2}\left[\frac{1}{|x|^\mu}\ast F(u_n)\right]F(u_n)
\end{eqnarray*}
so that
\begin{eqnarray*}
  \frac12\left(1-\frac 1K\right)\|u_n\|_W^2\leq \Phi_W(u_n)-\frac
1{2K}\langle\Phi'_W(u_n),u_n\rangle\leq \frac c2+\frac{\tau_n}{2K}\|u_n\|_W
\end{eqnarray*}
which implies that $\|u_n\|_W$ is bounded. As a consequence we have from \eqref{convPhi} and \eqref{convPhi'} that
\begin{equation}\label{bound}
\int_{\mathbb R^2}\left[\frac{1}{|x|^\mu}\ast
F(u_n)\right]F(u_n)\leq C, \quad \int_{\mathbb
R^2}\left[\frac{1}{|x|^\mu}\ast F(u_n)\right]u_nf(u_n)\leq C
\end{equation}
with $C$ independent of $n$. Moreover,  $u_n\rightharpoonup u$, $u_n\to u$ in $L^q_{loc}(\mathbb R^2)$ for any $1\leq q<\infty$ and
$u_n\to u$ a.e. in $\mathbb R^2$.

\noindent Next let us prove \eqref{convFF},
that is,
\begin{equation*}
\left|\int_{\Omega}\left[\frac{1}{|x|^\mu}\ast
F(u_n)\right]F(u_n)dx-\int_{\Omega}\left[\frac{1}{|x|^\mu}\ast
F(u)\right]F(u)dx\right|\to 0, \quad \forall\,\Omega \subset\subset \mathbb
R^2
\end{equation*}
This can be done as in \cite[Lemma 2.1]{DMR}. Indeed, since
$u\in H^1(\mathbb R^2)$, then $ \left[\frac{1}{|x|^\mu}\ast F(u)\right]F(u) \in
L^1(\mathbb R^2)$, so that
$$
\lim_{M\to \infty}\int_{\{u\geq M\}} \left[\frac{1}{|x|^\mu}\ast
F(u)\right]F(u)dx=0.
$$
Let $C$ be the constant in \eqref{bound} and  $M_0$ the constant in
$(f_1)$: for any $\delta>0$ we can choose $M>\max\{(CM_0/\delta)^{q+1}, s_0\}$ such
that
$$
0\leq\int_{\{u\geq M\}} \left[\frac{1}{|x|^\mu}\ast
F(u)\right]F(u)dx< \delta.
$$
From \eqref{bound} and $(f_1)(ii)$ we also have
$$
0\leq\int_{\{u_n\geq M\}}\left[\frac{1}{|x|^\mu}\ast
F(u_n)\right]F(u_n)dx\leq \frac{M_0}{M^{q+1}}\int_{\{u_n\geq
M\}}\left[\frac{1}{|x|^\mu}\ast F(u_n)\right]u_nf(u_n)dx< \delta,
$$
then we obtain
\begin{multline*}
\left|\int_{\Omega}\left[\frac{1}{|x|^\mu}\ast
F(u_n)\right]F(u_n)dx-\int_{\Omega}\left[\frac{1}{|x|^\mu}\ast
F(u)\right]F(u)dx\right|\leq \\2\delta+ \left|\int_{\Omega \cap
\{u_n\leq M\}}\left[\frac{1}{|x|^\mu}\ast
F(u_n)\right]F(u_n)dx-\int_{\Omega \cap \{u\leq
M\}}\left[\frac{1}{|x|^\mu}\ast F(u)\right]F(u)dx\right|.
\end{multline*}
 It remains then to prove that
\begin{equation}\label{equivts}
\int_{|u_n|\leq M}\left[\frac{1}{|x|^\mu}\ast
F(u_n)\right]F(u_n)\chi_{\Omega}dx\to \int_{|u|\leq
M}\left[\frac{1}{|x|^\mu}\ast F(u)\right]F(u)\chi_{\Omega}dx
\end{equation}
as $n\to +\infty$, for any fixed $M>\max\{(CM_0/\delta)^{q+1}, s_0\}$ . Let
us observe that as $K\to +\infty$
$$
\int_{|u|\leq M}\int_{|u|\leq K}\left[\frac{F(u(y))}{|x-y|^\mu}
\right]dyF (u(x))\chi_{\Omega}(x)dx \to \int_{|u|\leq
M}\left[\frac{1}{|x|^\mu}\ast F(u)\right]dy F(u)\chi_{\Omega}dx.
$$
Let  $C$  be the constant appearing in \eqref{bound} , and choose
$K\geq \max\{(CM_0/\delta)^{q+1}, s_0\}$ such that
$$
\int_{|u|\leq M}\int_{|u|\geq K}\left[\frac{F(u(y))}{|x-y|^\mu}
\right]dy F(u(x))dx\leq \delta.
$$
By $(f_1)(ii)$ one has
$$
\aligned
\int_{|u_n|\leq M}&\int_{|u_n|\geq
K}\left[\frac{F(u_n(y))}{|x-y|^\mu}
\right]F(u_n(x))\chi_{\Omega}(x)dx\\
&\leq \frac{1}{K^{q+1}}
\int_{|u_n|\leq M}\int_{|u_n|\geq K}\left[\frac{u_n^{q+1}
F(u_n)}{|x-y|^\mu}
\right]dy F(u_n)\chi_{\Omega}dx\\
&\leq \frac{M_0}{K^{q+1}} \int_{|u_n|\leq M}\int_{|u_n|\geq
K}\left[\frac{u_n f(u_n)}{|x-y|^\mu} \right]dy
F(u_n)\chi_{\Omega}dx \\
&\leq \frac{M_0}{K^{q+1}} \int_{|u_n|\leq M}\int_{|u_n|\geq
K}\left[\frac{u_n f(u_n)}{|x-y|^\mu} \right]dy F(u_n)dx\\
&=\frac{M_0}{K^{q+1}} \int_{\mathbb R^2}\int_{\mathbb
R^2}\left[\frac{F(u_n)}{|x-y|^\mu} \right]dy u_nf(u_n)dx\\
&\leq \delta,
\endaligned
$$
then we can see that
$$
\left|\int_{|u|\leq M}\int_{|u|\geq K}\left[\frac{F(u)}{|x-y|^\mu}
\right]dyF(u)\chi_{\Omega}-\int_{|u_n|\leq M}\int_{|u_n|\geq
K}\left[\frac{F(u_n)}{|x-y|^\mu} \right]dyF(u_n)\chi_{\Omega}
\right|\leq 2\delta.
$$
In order to prove \eqref{equivts} it remains to verify that as $n\to +\infty$ there holds
$$
\left|\int_{|u|\leq M}\int_{|u|\leq K}\left[\frac{F(u)}{|x-y|^\mu}
\right]dyF(u)\chi_{\Omega}-\int_{|u_n|\leq M}\int_{|u_n|\leq
K}\left[\frac{F(u_n)}{|x-y|^\mu} \right]dyF(u_n)\chi_{\Omega}
\right|\to 0
$$
for any fixed $K, M>0$. This is a consequence of
the Lebesgue's dominated convergence theorem: indeed,
$$\int_{|u_n|\leq K}\left[\frac{F(u_n)}{|x-y|^\mu}
\right]dy F(u_n)\chi_{\{\Omega\cap |u_n|\leq M\}}\to
\int_{|u|\leq K}\left[\frac{F(u)}{|x-y|^\mu}
\right]dy F(u)\chi_{\{\Omega\cap |u|\leq M\}} \quad \hbox{ a.e.}
$$
and by $(f_2)$ we know there exists a constant $C_{M,K}$ depends of $M,K$ such that
\begin{multline*}
\int_{|u_n|\leq K}\left[\frac{F(u_n)}{|x-y|^\mu} \right]dy
F(u_n)\chi_{\{\Omega\cap |u_n|\leq M\}}\\
\leq C_{M,K}
\int_{|u_n|\leq K}\left[\frac{u_n^{p+1}}{|x-y|^\mu}
\right]dy u_n^{p+1}\chi_{\{\Omega\cap |u_n|\leq M\}}\\
\leq  C_{M,K} \int_{\mathbb R^2}\left[\frac{1}{|x|^\mu}\ast
u_n^{p+1}\right] u_n^{p+1}\chi_{\Omega}
\to C_{M,K} \int_{\mathbb R^2} \left[\frac{1}{|x|^\mu}\ast
u^{p+1}\right] u^{p+1}\chi_{\Omega}
\end{multline*}
as $n\to\infty$, applying the Hardy-Sobolev-Littlewood
inequality, since $u_n\to u$ in $L^s_{loc}$ for all $s\geq 1 $. Hence the proof of \eqref{convFF} is now complete.

Let us now prove that the weak limit $u$ yields actually a weak solution to
\eqref{A}, namely that
\begin{equation}\label{weaklim}
\int_{\mathbb R^2}\nabla u\nabla
\varphi +W(x)u\varphi-\left[\frac{1}{|x|^{\mu}}\ast F(u)\right]f(u)\varphi=0
\end{equation}
for all $\varphi \in \mathcal C^{\infty}_c(\mathbb R^2)$. Since $\{u_n\}$ is a
$(PS)_{m_V}$ sequence, for all $\varphi \in \mathcal C^{\infty}_c(\mathbb R^2)$, we know that
$$
\int_{\mathbb R^2}\nabla u_n\nabla \varphi
+W(x)u_n\varphi-\left[\frac{1}{|x|^{\mu}}\ast
F(u_n)\right]f(u_n)\varphi\to 0,
$$
 as $n\to\infty$. Since $u_n\rightharpoonup u$ in $E$, we just need to prove that, as $n\to\infty$
\begin{equation}\label{weak*conv}
\int_{\mathbb R^2}\left[\frac{1}{|x|^{\mu}}\ast
F(u_n)\right]f(u_n)\varphi \rightarrow \int_{\mathbb
R^2}\left[\frac{1}{|x|^{\mu}}\ast F(u)\right]f(u)\varphi
\end{equation}
for all $\varphi \in C^{\infty}_c(\mathbb R^2)$.

Let $\Omega$ be
any compact subset of $\mathbb R^2$, we claim that there exists  $C(\Omega)$ such that
\begin{equation}\label{boundbis}
   \int_{\Omega}\left[\frac{1}{|x|^\mu}\ast
F(u_n)\right]\frac{f(u_n)}{1+u_n}dx\leq C(\Omega).
\end{equation}
 In fact,
let
$$
v_n=\frac{\varphi}{1 +u_ n},
$$
where $\varphi$ is a smooth function compactly supported in $\Omega'\supset \Omega$, $\Omega'$ compact,
such that $0\leq\varphi\leq 1$ and $\varphi \equiv 1$ in $\Omega$. Direct computation shows that
$$
\aligned
\|v_n\|_W^2&=\int_{\mathbb R^2}\left|\nabla
v_n\right|^2+W(x)v_n^2\\
&=\int_{\mathbb R^2}\left|\frac{\nabla
\varphi}{1+u_n}-\varphi\frac{\nabla
u_n}{(1+u_n)^2}\right|^2+W\frac{\varphi ^2}{(1+u_n)^2}\\
&\leq
\int_{\mathbb R^2}\frac{|\nabla
\varphi|^2}{(1+u_n)^2}+2\frac{\nabla \varphi \nabla
u_n}{1+u_n}+\varphi^2\frac{|\nabla u_n|^2}{(1+u_n)^4}+W\varphi
^2\\
&\leq 2\|\varphi\|_W^2 + 2\|u_n\|_W^2,
\endaligned
$$
which means that $v_n\in E$. Choose $v_n$ as test function in \eqref{convPhi'}, then
$$
\aligned
 \int_{\Omega}\left[\frac{1}{|x|^\mu}\ast
F(u_n)\right]\frac{f(u_n)}{1+u_n}dx&\leq\int_{\mathbb
R^2}\left[\frac{1}{|x|^\mu}\ast
F(u_n)\right]f(u_n)\frac{\varphi}{1+u_n}\\
&\leq \int_{\mathbb
R^2}|\nabla u_n|^2 \frac{\varphi}{(1+u_n)^2}+\frac{\nabla u_n\nabla\varphi}{1+u_n}+Wu_n\frac{\varphi}{1+u_n}+ \tau_n\|v_n\|_W\\
&\leq \int_{\mathbb R^2}|\nabla
u_n|^2 \frac{\varphi}{(1+u_n)^2}+\frac{\nabla
u_n\nabla\varphi}{1+u_n}+Wu_n\frac{\varphi}{1+u_n}+
2\tau_n\|u_n\|_W+2\tau_n\|\varphi\|_W\\
&\leq \|\nabla u_n\|_2^2 +C_{\varphi}\|\nabla
u_n\|_2+\int_{\Omega'}Wu_n+2\tau_n\|u_n\|_W+2\tau_n\|\varphi\|_W.
\endaligned
$$
Since $W(x)$ is bounded,
$u_n$ is bounded in $H^1$ and $u_n\to u $ in $L^1(\Omega')$  we
easily deduce \eqref{boundbis}.

Now define
$$\xi_n:=\left[\frac{1}{|x|^{\mu}}\ast F(u_n)\right]f(u_n),$$
we can observe that
$$
\aligned
\int_{\Omega}&\left[\frac{1}{|x|^\mu}\ast F(u_n)\right]f(u_n)dx\\
&\leq
2\int_{\{u_n<1\}\cap \Omega}\left[\frac{1}{|x|^\mu}\ast
F(u_n)\right]\frac{f(u_n)}{1+u_n}dx+\int_{\{u_n>1\}\cap
\Omega}\left[\frac{1}{|x|^\mu}\ast
F(u_n)\right]u_nf(u_n)dx\\
&\leq 2\int_{\Omega}\left[\frac{1}{|x|^\mu}\ast
F(u_n)\right]\frac{f(u_n)}{1+u_n}dx+\int_{\mathbb
R^2}\left[\frac{1}{|x|^\mu}\ast F(u_n)\right]u_nf(u_n)dx.
\endaligned
$$
Combining \eqref{boundbis} and \eqref{bound}, it is easy to see that $\xi_n$ is uniformly bounded in
$L^1(\Omega)$ with
$$
\int_{\Omega}\left[\frac{1}{|x|^\mu}\ast F(u_n)\right]f(u_n)dx\leq 2C(\Omega)+C.
$$
Finally, consider the sequence of measures $\mu_n$ with density $\xi_n=\left[\frac{1}{|x|^{\mu}}\ast
F(u_n)\right]f(u_n)$, that is
$$
\mu_n(E):=\int_E\xi_n \,dx = \int_E \left[\frac{1}{|x|^{\mu}}\ast
F(u_n)\right]f(u_n) \,dx \quad \hbox{ for any measurable }
E\subset \Omega
$$
Since $\|\xi_n\|_1\leq C(\Omega)$ and $\Omega$ is bounded, the
measures $\mu_n$ have uniformly bounded total variation. Then, by
weak$^\ast$-compactness, up to a subsequence, $\mu_n
\rightharpoonup^\ast \mu$ for some measure $\mu$,
$$
\lim_{n\to\infty} \int_{\Omega} \xi_n \varphi\,dx=\lim_{n\to\infty}  \int_{\Omega}
\left[\frac{1}{|x|^{\mu}}\ast F(u_n)\right]f(u_n)\varphi \,dx
=\int_{\Omega}\varphi d\mu, \quad \forall\: \varphi \in C^{\infty}_c(\Omega).
$$
Now recall that $u_n$ is a (PS) sequence, so that in particular
\eqref{convPhi'} holds and hence
$$
\lim_{n\to\infty} \int_{\mathbb R^2}\nabla u_n\nabla \varphi
+W(x)u_n\varphi=\int_{\Omega}
\varphi d\mu ,\quad \forall\: \varphi \in C^{\infty}_c(\Omega),
$$
which implies that $\mu$ is absolutely
continuous with respect to the Lebesgue measure. Then, by the
Radon-Nicodym theorem, there exists a function $\xi \in
L^1(\Omega)$ such that
$$
\int_{\Omega} \varphi d\mu=\int_{\Omega} \varphi \xi dx, \quad
\forall\:  \varphi \in C^{\infty}_c(\Omega).
$$
Since this holds for any compact set $\Omega\subset \mathbb R^2$, we have that there exists a function $\xi \in
L^1_{loc}(\mathbb R^2)$ such that
$$
\int_{\mathbb R^2} \varphi d\,\mu=\lim_n\int_{\mathbb R^2}
\left[\frac{1}{|x|^{\mu}}\ast F(u_n)\right]f(u_n)\varphi
\,dx=\int_{\mathbb R^2} \varphi \xi d\,x, \quad \forall\:  \varphi \in C^{\infty}_c(\mathbb R^2),
$$
where $\xi=\left[\frac{1}{|x|^{\mu}}\ast F(u)\right]f(u)$ and the proof is complete.
%


\end{proof}
\noindent
{\bf Proof of Theorem \ref{thm-Existence}.}
As proved in \cite[Lemma 2.1]{AY2}, the functional $\Phi_W$
satisfies the Mountain Pass geometry, then there exists a
$(PS)_{m_W}$ sequence $\{u_n\}$. By Lemma \ref{lem-PS}, up to a
subsequence, $\{u_n\}$ weakly converges to a weak solution $u$ of
\eqref{A}: it remains only to prove that $u$ is non-trivial. Let
us suppose by contradiction that $u\equiv 0$. Since $\{u_n\}$ is
bounded, we have either $\{u_n\}$ is vanishing, that is, for any
$r>0$
$$
\lim_{n\to +\infty}\sup_{y\in \mathbb R^2}\int_{B_r(y)}|u_n|^2=0
$$
or it is non-vanishing, i.e. there exist $r, \delta >0$ and a
sequence $\{y_n\}\subset \mathbb Z^2$ such that
$$
\lim_{n\to \infty}\int_{B_r(y_n)}|u_n|^2\geq \delta
$$
If $\{u_n\}$ is vanishing, by Lions' concentration-compactness result we have
\begin{equation}\label{Lions}
u_n\to 0 \quad \mbox{in} \quad  L^s(\mathbb R^2) \quad \forall \, s>2,
\end{equation}
as $n\to\infty$. In this case we claim that
\begin{equation}\label{conv0}
\left[\frac{1}{|x|^{\mu}}\ast F(u_n)\right]F(u_n)\to0
\quad \hbox{ in } L^1(\mathbb R^2),
\end{equation}
as $n\to\infty$. In fact, we need only to repeat the proof of \eqref{convFF} in Lemma
\ref{lem-PS} without restricting necessarily to compact sets. Apply the Hardy-Sobolev-Littewood inequality we notice that
$$
\left|\int_{\mathbb R^2} \left[\frac{1}{|x|^\mu}\ast
u_n^{p+1}\right] u_n^{p+1}
\right| \leq C|u_n|^{2(p+1)}_{\frac{4}{4-\mu}(p+1)}\to 0
$$
as $n\to\infty$,  since $\frac{4}{4-\mu}(p+1)>2$ and \eqref{Lions} holds. Since
$\{u_n\}$ is a $(PS)_{m_W}$ sequence with $m_W<
\frac{4-\mu}{8},$ it follows that
$$
\lim_{n\to +\infty}\|u_n\|_W^2=2m_W<\frac{4-\mu}{4}
$$
Then there exist a sufficiently small $\delta>0$ and $K>0$ such that
\begin{equation}\label{conv-norm}
\|u_n\|_W^2\leq \frac{4-\mu}{4}(1-\delta), \qquad \forall\, n>K.
\end{equation}
Using again the Hardy-Sobolev-Littewood inequality we have
$$
\left|\int_{\mathbb R^2}\left[\frac{1}{|x|^{\mu}}\ast
F(u_n)\right]f(u_n)u_n\right|\leq C|F(u_n)|_{\frac{4}{4-\mu}}
|f(u_n)u_n|_{\frac{4}{4-\mu}}.
$$
Combining $(f_1)$ with $(f_2)$, for any $\vr>0$, $p> 1$ and $\beta>1$, there exists $C(\vr,p, \beta)>0$ such that
$$
|f(s)|\leq \vr |s|^{\frac{2-\mu}{2}}+C(\vr,p, \beta) |s|^{p-1}\big[ e^{ \beta4\pi s ^{2}}
-1 \big] \,\,\, \forall s\in \R.
$$
Then,
$$
|f(u_n)u_n|_{\frac{4}{4-\mu}}\leq \vr |u_n|_2^{\frac{4-\mu}{2}}+C(\vr,p, \beta)|u_n|_{\frac{4p t'}{4-\mu}}^{\frac{4-\mu}{4t'}}\big(\int_{\R^2}[e^{( \frac{4\beta t}{4-\mu}\|u_n\|^2_W 4\pi \frac{u_n ^{2}}{\|u_n\|^2_W})}
-1]\big)^{\frac{4-\mu}{4t}}
$$
where $t, t'>1$ satisfying $\frac{1}{t}+\frac{1}{t'}=1$. In order to conclude by means of  \cite{O} by do \'{O} and Adachi-Tanaka inequality \cite{AT}
it is enough to
choose $\beta, t>1$ close to $1$ such that $\frac{4\beta t}{4-\mu}\|u_n\|^2_W<1$, namely
$$
1< \beta t < \frac 1{1-\delta},
$$
we deduce that
$$
\big(\int_{\R^2}[e^{( \frac{4\beta t}{4-\mu}\|u_n\|^2_W 4\pi \frac{u_n ^{2}}{\|u_n\|^2_W})}
-1]\big)^{\frac{4-\mu}{4t}}\leq \big(\int_{\R^2}[e^{( \frac{4\beta m t}{4-\mu} 4\pi \frac{u_n ^{2}}{\|u_n\|^2_W})}
-1]\big)^{\frac{4-\mu}{4t}} \leq C_1 \,\,\,\, \forall n>K,
$$
for some $C_1>0$. Then,
$$
\aligned
\Big|\int_{\mathbb R^2}\left[\frac{1}{|x|^{\mu}}\ast
F(u_n)\right]f(u_n)u_n \Big|&\leq \vr^2 |u_n|_2^{{4-\mu}}+C_2|u_n|_{\frac{4p t'}{4-\mu}}^{\frac{4-\mu}{2t'}}.
\endaligned
$$
Since $t>1$ is close to $1$, we have that $\frac{4p t'}{4-\mu}>2$. By \eqref{Lions}, we have
$$
\left|\int_{\mathbb R^2}\left[\frac{1}{|x|^{\mu}}\ast
F(u_n)\right]f(u_n)u_n\right| \to 0
$$
as $n\to\infty$. Recalling that $\{u_n\}$ is a $(PS)_{m_W}$ sequence, $u_n
\to 0$ in $E$, and so
 $\Phi_W(u_n)\to 0$
which implies $m_W=0$, which is a contradiction. Therefore the
vanishing case dose not hold.\\

\noindent Let us now consider the non vanishing case and define
$v_n:=u_n(\cdot-y_n)$, then
\begin{equation}\label{nonvan}
\int_{B_r(0)}|v_n|^2\geq \delta
\end{equation}
By the periodicity assumption, $\Phi_W$ and $\Phi_{W'}$ are both
invariant by $\mathbb Z^2$ translations, so that $\{v_n\}$ is
again a $(PS)_{m_W}$ sequence. Then $v_n\rightharpoonup v$ in $E$,
with $v\neq 0$ by
using \eqref{nonvan}, since $v_n\to v $ in $L^2_{loc}(\mathbb R^2)$ . Thereby, $v$ is a nontrivial critical point of $\Phi_W$ and $\Phi_W(v)=m_W$, which completes the proof of the theorem.

\section{Semiclassical states for the nonlocal Schr\"{o}dinger equation}

Performing the scaling $u(x)=v(\epsilon x)$ one easily sees that problem \eqref{EC} is equivalent to
$$
\begin{array}{l}
\displaystyle -\Delta u +V(\varepsilon x)u  =\Big[\frac{1}{|x|^{\mu}}\ast F(u)\Big]f(u).
\end{array}
\eqno{(SNS^*)}
$$

For $\vr>0$, we define the following Hilbert space
 $$
E_{\vr}=\Big\{u\in E:   \int_{\R^2}V(\vr x)|u|^2<\infty\Big\}
$$
endowed with the norm
 $$
 \|u\|_{\vr}:=\left(\int_{\R^2}\big(|\nabla u|^2+ V(\vr x)|u|^2\big)\right)^{1/2}.
 $$

The energy functional associated to equation $(SNS^*)$ is given by
$$
\aligned
I_{\vr}(u)=\frac12\|u\|_{\varepsilon}^2-\mathfrak{F}(u)
\endaligned
$$
and
$$
\langle I_{\vr}'(u), \varphi\rangle=\int_{\mathbb{R}^2} (\nabla u\nabla \varphi+ V(\vr x) u\varphi)-\mathfrak{F}'(u)[\varphi], \,\,\, \forall u, \varphi \in E.
$$

\noindent Let ${\cal{N}_\vr}$ be the Nehari manifold associated to $I_{\vr}$, that is,
$$
{\cal{N}_\vr}=\Big\{u\in E_\vr:u\neq0, \langle I_{\vr}'(u),u\rangle=0\Big\}.
$$
The following Lemma tells that the Nehari manifold ${\cal{N}_\vr}$ is bounded away from $0$.
\begin{lem}\label{LN}
Suppose that conditions $(f_1)-(f_3)$ hold. Then there exists $\alpha>0$, independent of $\vr$, such that
\begin{equation} \label{alpha2}
\|u\|_{\vr}\geq \alpha, \,\,\, \forall u\in \cal{N}_\vr.
\end{equation}
\end{lem}
\begin{proof}
 For any $\delta>0$, $p>1$ and $\beta>1$, there exists $C_\delta>0$ such that
$$
F(s)<\frac{1}{K}f(s)s\leq\delta s^{\frac{4-\mu}{2}}+C(\delta,p, \beta) s^{p}\big[ e^{ \beta4\pi s ^{2}}
-1 \big], \forall s\in \R,
$$
it follows
\begin{equation} \label{mp1}
|F(u)|_{\frac{4}{4-\mu}}\leq C|f(u)u|_{\frac{4}{4-\mu}}\leq \delta C |u|^{\frac{4-\mu}{2}}_2+C(\delta, p, \beta)\big| u^{p}\big[ e^{ \beta4\pi u ^{2}}
-1 \big]\big|_{\frac{4}{4-\mu}}.
\end{equation}
Since the imbedding $E_{\vr}\hookrightarrow L^p(\R^2)$ is continuous for any $p\in (2,+\infty)$, we know there exists a constant $C_1$ such that
$$\aligned
\int_{\R^2}|u|^{\frac{4p}{4-\mu}}\big[ e^{\beta4\pi u ^{2}}
-1 \big]^\frac{4}{4-\mu}&\leq (\int_{\R^2}|u|^{\frac{8p}{4-\mu}})^{\frac12}(\int_{\R^2}\big[ e^{ \beta4\pi u ^{2}}
-1 \big]^\frac{4}{4-\mu})^{\frac12}\\
&\leq C_1\|u\|_{\vr}^{\frac{4p}{4-\mu}}\big(\int_{\R^2}\big[ e^{(\frac{4\beta}{4-\mu}4\pi u ^{2})}
-1 \big]\big)^{\frac12}.
\endaligned
$$
Notice that
$$
\int_{\R^2}\big[ e^{(\frac{4\beta}{4-\mu}4\pi u ^{2})}
-1 \big]=\int_{\R^2}\big[ e^{(\frac{4\beta}{4-\mu}\|u\|_{\vr}^2 4\pi \frac{u ^{2}}{\|u\|_{\vr}^2})}
-1 \big],
$$
then, fixing $\xi \in (0,1)$ and making $\frac{4\beta}{4-\mu}\|u\|_{\vr}^2=\xi< 1$, Lemma \ref{Trudinger-Moser} implies that there exists a constant $C_2$ such that
$$
\int_{\R^2}\big[ e^{(\xi 4\pi \frac{u ^{2}}{\|u\|_{\vr}^2})}
-1 \big]\leq C_2.
$$
thus, by \eqref{mp1}, we know there exists $C_3$ such that
$$
|F(u)|_{\frac{4}{4-\mu}}\leq \delta \|u\|_{\vr}^{\frac{4-\mu}{2}}+C_3\|u\|_{\vr}^{p}.
$$
By Hardy-Littlewood-Sobolev inequality, if $\|u\|_{\vr}^2 = \frac{\xi (4-\mu)}{4\beta}$, there holds
$$
\mathfrak{F}'(u)[u]\leq \delta^2 C_4 \|u\|_{\vr}^{4-\mu}+C_4\|u\|_{\vr}^{2p}.
$$

Since $u\in {\cal N}_{\vr}$, there holds
$$
\|u\|_{\vr}^2=\mathfrak{F}'(u)[u],
$$
and so
$$
\|u\|_{\vr}^{2}\leq \delta^2 C_5\|u\|_{\vr}^{4-\mu}+C_5\|u\|_{\vr}^{2p},
$$
then the conclusion follows immediately.
\end{proof}

\noindent Next we show that the functional $I_{\vr}$ satisfies the Mountain Pass Geometry.
\begin{lem}\label{mountain:1}
Suppose that conditions $(f_1)-(f_3)$ hold, then
\begin{itemize}
  \item[$(i)$] There exist $\rho, \delta_0>0$ such that $I_{\vr}|_S\geq\delta_0>0$ for all $u\in S=\{u\in E_{\vr}:\|u\|_{\vr}=\rho\}$;
  \item[$(ii)$]There is $e$ with $\|e\|_{\vr}>\rho$ such that $I_{\vr}(e)< 0$.
\end{itemize}
\end{lem}
\begin{proof}
The proof of $(i)$ easily follows buying the line of Lemma \ref{LN}, so that we only prove $(ii)$. Fixed $u_0 \in E_{\vr}$ with $u_0^{+}(x)=\max\{u_0(x),0\}$, we set
$$
w(t)=\mathfrak{F}(\frac{tu_0}{\|u_0\|_{\vr}})>0, \,\,\ \mbox{for} \,\,\, t>0.
$$
By the Ambrosetti-Rabinowitz condition $(f_3)$ we know
$$
\frac{w'(t)}{w(t)}\geq \frac{2K}{t} \,\,\, \mbox{for} \,\,\, t>0.
$$
Integrate this over $[1, s\|u_0\|_{\vr}]$ with $s>\frac{1}{\|u_0\|_{\vr}}$ to get
$$
\mathfrak{F}(su_0)\geq \mathfrak{F}(\frac{u_0}{\|u_0\|_{\vr}})\|u_0\|_{\vr}^{2K} s^{2K}.
$$
Therefore
$$
I_{\vr}(su_0)\leq C_1 s^2-C_2s^{2K} \,\,\, \mbox{for} \,\,\ s > \frac{1}{\|u_0\|_{\vr}}.
$$
Since $K>1$, $(ii)$ follows taking $e=s u_0$ and $s$ large enough.
\end{proof}

\noindent By the Ekeland Variational Principle \cite{E} we know there is a $(PS)_{c_{\vr}}$ sequence $(u_n) \subset E$, i.e.
$$
I_{\vr}'(u_n)\rightarrow0,\quad I_{\vr}(u_n)\rightarrow
c_{\vr},
$$
where $c_{\vr}$ defined by
\begin{equation} \label{m1}
0<c_{\vr}:=\inf_{u\in E\backslash\{0\}} \max_{t\geq 0}
I_{\vr}(tu)
\end{equation}
and moreover there is a constant $c>0$ independent of $\vr$ such that $c_{\vr}>c>0$.
Using assumption  $(f_5)$, for each $u\in E_{\vr}\backslash\{0\}$, there is an unique $t=t(u)$ such that
 $$
I_{\vr}(t(u)u)=\max_{s\geq0}I_{\vr}(su)\ \ \hbox{and}\ \ t(u)u\in {\cal N}_{\vr}.
 $$
Then it is standard to see (see \cite{MW}) that the minimax value $c_{\vr}$ can be characterized by
\begin{equation} \label{m2}
c_{\vr}= \inf_{u\in {\cal N}_{\vr}}I_{\vr}(u).
\end{equation}

\begin{lem}\label{EML}
Suppose that assumptions $(f_1)-(f_5)$, $(V_1)$ and $(V_2)$ hold. Let $c_{\vr}$ be the minimax value defined in \eqref{m1}, then there holds
$$
\lim_{\vr\to0}c_{\vr}=m_{V_{0}},
$$
where $m_{V_{0}}$ is the minimax value defined in \eqref{m} with $W(x)\equiv V_{0}$.
Hence, by Lemma \ref{MPlevel-estimate}, there is $\vr_0>0$ such that
$$
c_\vr <\frac{4-\mu}{8}, \quad \forall \vr \in [0, \vr_0).
$$
Moreover, since $m_{V_{0}} < m_{V_{\infty}}$, we also have
$$
\lim_{\vr\to0}c_{\vr}\leq m_{V_{\infty}}.
$$
\end{lem}
\begin{proof}
Let $w\in E$ be the ground state solution obtained in Theorem \ref{thm-Existence}, then there holds
$$
\int_{\R^2}\big(|\nabla w|^2+ V_0|w|^2\big)=\int_{\R^2}\Big[\frac{1}{|x|^{\mu}}\ast F(w)\Big]f(w)w
$$
In what follows, given $\delta >0$, we fix $w_\delta \in C_{0}^{\infty}(\mathbb{R}^{2})$ verifying
\begin{equation} \label{ESc1}
w_\delta \in {\cal N}_{V_0}, \,\, w_\delta \to w \,\,\, \mbox{in} \,\,\, E \,\,\, \mbox{and} \,\,\, \Phi_{V_0}(w_\delta) < m_{V_0} +\delta.
\end{equation}
Now, choose $\eta\in C^{\infty}_0(\R^2, [0,1])$ be such that $\eta=1$ on $B_1(0)$ and $\eta=0$ on $\R^2\backslash B_2(0)$, let us define $v_{n}(x)=\eta(\vr_n x)w_\delta(x)$, where $\vr_n \to 0$. Clearly
$$
v_{n}\to w_\delta \ \ \hbox{in} \ \ E,\ \ \hbox{as} \ \ n\to +\infty.
$$
From the definition of $ {\cal N}_{\vr}$, we know that there exists unique $t_{n}$ such that $t_{n}v_{n}\in {\cal N}_{\vr_n}$. Consequently,
$$
c_{\vr_n}\leq I_{\vr_n}(t_{n}v_{n})=\frac{t^2_{n}}{2}\int_{\R^2}\big(|\nabla v_{n}|^2+ V(\vr_n x)|v_{n}|^2\big)-\frac12\int_{\R^2}\Big[\frac{1}{|x|^{\mu}}\ast F(t_{n}v_{n})\Big]F(t_{n}v_{n}).
$$
Observe that
$$
\langle I_{\vr_n}'(t_{n}v_{n}),t_{n}v_{n}\rangle=0,
$$
or equivalently,
\begin{multline} \label{ESc}
{t^2_{n}}\int_{\R^2}\big(|\nabla v_{n}|^2+ V(\vr_n x)|v_{n}|^2\big)=\int_{\R^2}\Big[\frac{1}{|x|^{\mu}}\ast F(t_{n}v_{n})\Big]f(t_{n}v_{n})t_{n}v_{n}\\
\geq C t^{2K}_{n}\int_{\R^2}\Big[\frac{1}{|x|^{\mu}}\ast |v_{n}|^{K}\Big]|v_{n}|^{K}
\end{multline}
which means $\{t_{n}\}$ is bounded and thus, up to  subsequence, we may assume that $t_{n}\to t_0\geq 0$. Notice that there is a constant $c>0$ independent of $\vr$ such that $c_{\vr_n}>c>0$. Then, this information implies that $t_0> 0$. Take limit in the equality in \eqref{ESc} to find
\begin{equation} \label{ESc2}
\int_{\R^2}\big(|\nabla w_\delta|^2+ V_0|w_\delta|^2\big)=t_0^{-2}\int_{\R^2}\Big[\frac{1}{|x|^{\mu}}\ast F(t_0w_\delta)\Big]f(t_0w_\delta)t_0w_\delta.
\end{equation}
Hence, from \eqref{ESc1} and \eqref{ESc2},
$$
t_0^{-2}\int_{\R^2}\Big[\frac{1}{|x|^{\mu}}\ast F(t_0w)\Big]f(t_0w)t_0w-\int_{\R^2}\Big[\frac{1}{|x|^{\mu}}\ast F(w)\Big]f(w)w=0.
$$
Thereby, by monotone assumption $(f_5)$, we derive that
$$
t_0=1.
$$
Since
$$
\int_{\R^2}\big(V(\vr_n x)-V_0\big)|v_{n}|^2 \to 0 \,\,\ \mbox{and} \,\,\, \Phi_{V_{0}}(t_{n}v_{n}) \to \Phi_{V_{0}}(w_\delta),
$$
the following inequality
$$
c_{\vr_n}\leq I_{\vr_n}(t_{n}v_{n})=\Phi_{V_{0}}(t_{n}v_{n})+\frac{t^2_{n}}{2}\int_{\R^2}\big(V(\vr_n x)-V_0\big)|v_{n}|^2,
$$
gives
$$
\limsup_{n \to +\infty} c_{\vr_n}\leq \Phi_{V_{0}}(w_\delta) \leq m_{V_0} +\delta.
$$
As $\delta$ is arbitrary, we deduce that
$$
\limsup_{n \to +\infty} c_{\vr_n}\leq m_{V_0}.
$$
As $\vr_n$ is also arbitrary, it follows that
\begin{equation} \label{PASSO1}
\limsup_{\vr \to 0} c_{\vr}\leq m_{V_0}.
\end{equation}

\noindent  On the other hand, we already know that
$$
c_\vr \geq m_{V_0}, \quad \forall \vr >0,
$$
which implies
\begin{equation} \label{PASSO2}
\liminf_{\vr \to 0}c_\vr \geq m_{V_0}.
\end{equation}
From (\ref{PASSO1}) and (\ref{PASSO2}) we get
$$
\lim_{\vr \to 0}c_\vr \geq m_{V_0}.
$$
and the proof follows by using Lemma \ref{MPlevel-estimate}.

\end{proof}

\begin{lem}\label{PS}
Suppose that the assumptions $(f_1)-(f_5)$,  $(V_1)$ and $(V_2)$ hold. Let $\{u_n\}$ be a $(PS)_{c_{\vr}}$ sequence
 with $\vr \in [0, \vr_0)$. Let $ u_{\vr}$ be the weak limit of  $u_n$, then $\{u_n\}$ converges strongly to $u_\vr$ in $E_{\vr}$, i.e. $I_{\vr}$  satisfies $(PS)_{c_\vr}$ condition for $\vr \in [0, \vr_0)$.
\end{lem}
\begin{proof} First recall that
\begin{eqnarray}
&&c_\vr <\frac{4-\mu}{8}, \quad \forall \vr \in [0, \vr_0) \label{EST1}\\
&&m_{V_0} < m_{V_\infty}. \label{EST2}
\end{eqnarray}
and there are positive constants $a_1, a_2$ such that
\begin{equation}\label{EST3}
a_1<\|u_n\|_{\vr}<a_2, \quad \forall n \in \mathbb{N} \quad ( \mbox{for some subsequence}).
\end{equation}

In the sequel, our first goal is to prove that $u_\vr \not= 0$. To do that, we will argue by contradiction, assuming that $u_\vr=0$.

\vspace{0.5 cm}
\noindent {\bf Claim:} There exist $\beta,\tilde{R} >0$ and  $\{y_n\}\subset \mathbb{R}^2$ such that
$$
\int_{B_{\tilde{R}}(y_n)}|u_n|^{2}\geq \beta.
$$
Indeed, if not by applying a result due to Lions, we obtain
$$
u_n \to 0 \quad \mbox{in} \quad L^{q}(\mathbb{R}^{2}) \quad \forall q \in (2,+\infty).
$$
Following line by line the argument of Section 2,  we have
$$
\left|\int_{\mathbb R^2}\left[\frac{1}{|x|^{\mu}}\ast
F(u_n)\right]F(u_n)\right| \to 0, \quad n\to \infty.
$$
Since $(u_n)$ be a $(PS)_{c_{\vr}}$ sequence
 with $c_\vr <\frac{4-\mu}{8}$, we know that
\begin{equation}\label{EST4}
\limsup_{n \to \infty}\|u_n\|^{2}_{\vr}=2c_\vr <\frac{4-\mu}{4}.
\end{equation}
As in the proof of Theorem \ref{thm-Existence}, we can conclude that
$$
\left|\int_{\mathbb R^2}\left[\frac{1}{|x|^{\mu}}\ast
F(u_n)\right]f(u_n)u_n\right| \to 0, \quad n\to \infty.
$$
This together with $\langle I'_{\vr}(u_n), u_n\rangle=o_n(1)$ implies that
$$
\lim_{n \to +\infty}\|u_n\|^{2}_{\vr}=0
$$
which contradicts (\ref{EST4}), proving the claim.

Next, we fix $t_n >0$ such that $t_nu_n \in {{\cal N}_{V_{\infty}}}$. We claim that $\{t_n\}$ is bounded. In fact, setting  $v_n=u_n(x+y_n)$, by Claim 1, we may assume that, up to a subsequence, ${v}_n\rightharpoonup {v}$ in $E_\vr$. Moreover, using the fact that $u_n \geq 0$ for all $n \in \mathbb{N}$, there exists $a_3>0$ and a subset $\Omega\subset \mathbb{R}^2$ with positive measure such that ${v}(x) >a_3 $ for all $x \in \Omega$. We have
$$
\int_{\mathbb{R}^{2}}(|\nabla u_n|^{2}+V_\infty|u_n|^{2})=\int_{\mathbb{R}^2}\int_{\mathbb{R}^2}\Big(\frac{F(t_nu_n(y))f(t_nu_n(x))t_nu_n(x)}{t^2_n|x-y|^{\mu}}\Big)
$$
and so,
$$
\int_{\mathbb{R}^{2}}(|\nabla u_n|^{2}+V_\infty|u_n|^{2})=\int_{\mathbb{R}^2}\int_{\mathbb{R}^2}\Big(\frac{F(t_nv_n(y))f(t_nv_n(x))t_nv_n(x)}{t^2_n|x-y|^{\mu}}\Big)
$$
from which
$$
\int_{\mathbb{R}^{2}}(|\nabla u_n|^{2}+V_\infty|u_n|^{2})\geq \int_{\Omega}\int_{\Omega}\Big(\frac{F(t_nv_n(y))f(t_nv_n(x))t_nv_n(x)}{t^2_n|x-y|^{\mu}}\Big)
$$
Since
$$
\liminf_{n \to \infty}\frac{F(t_nv_n(y))f(t_nv_n(x))t_nv_n(x)}{t^2_n|x-y|^{\mu}}=+\infty \quad \mbox{a.e.}
$$
Fatou's lemma gives
$$
\liminf_{n \to +\infty}\int_{\mathbb{R}^{2}}(|\nabla u_n|^{2}+V_\infty|u_n|^{2})=+\infty,
$$
which is a contradiction since $\{u_n\}$ is bounded in $E_\vr$.  Thus, without loss of generality we may assume
$$
\lim_{n \to +\infty}t_n=t_0>0.
$$

\noindent In what follows, we divide the remaining part of the proof into three steps.\\
\noindent {\bf Step 1.} The number $t_0$ is less or equal to 1.

\noindent In fact, suppose by contradiction that the above claim does not hold. Then, there exist $\delta>0$ and a subsequence of $(t_n)$, still denoted by itself, such that
$$
t_n\geq 1+\delta\ \hbox{ for all}\ \ n\in \N.
$$
Since $\langle I'_{\vr}(u_n), u_n\rangle=o_n(1)$ and $(t_nu_n)\subset {{\cal N}_{V_{\infty}}}$, we have
$$
\int_{\mathbb{R}^2} (|\nabla u_n|^{2}+ V(\varepsilon x) |u_n|^{2})=\mathfrak{F}'(u_n)[u_n]+o_n(1)
$$
and
$$
t^2_n\int_{\mathbb{R}^2} (|\nabla u_n|^{2}+ V_{\infty} |u_n|^{2})=\mathfrak{F}'(t_nu_n)[t_nu_n].
$$
Consequently,
\begin{multline*}
\int_{\mathbb{R}^2} (V_{\infty}-V(\vr x))|u_n|^{2}+o_n(1)\\
=\int_{\mathbb{R}^2}\int_{\mathbb{R}^2}\Big(\frac{F(t_nu_n(y))f(t_nu_n(x))t_nu_n(x)}{t^2_n|x-y|^{\mu}}-\frac{F(u_n(y))f(u_n(x))u_n(x)}{|x-y|^{\mu}}\Big).
\end{multline*}
Given $\zeta>0$, from assumptions $(V_1)$ and $(V_2)$, there exists $R=R(\zeta)>0$ such that
\begin{equation}\label{V1}
V(\varepsilon x)\geq V_{\infty}-\zeta,\  \hbox{for any}\ |x|\geq R.
\end{equation}
Using the fact that $u_n\to 0$ in $L^2(B_R(0))$, we conclude that
$$
\int_{\mathbb{R}^2}\int_{\mathbb{R}^2}\Big(\frac{F(t_nu_n(y))f(t_nu_n(x))t_nu_n(x)}{t^2_n|x-y|^{\mu}}-\frac{F(u_n(y))f(u_n(x))u_n(x)}{|x-y|^{\mu}}\Big)\leq\zeta C+o_n(1),
$$
where $\displaystyle C=\sup_{n \in \mathbb{N}}|u_n|^2_2$. Using the sequence ${v}_n=u_n(x+y_n)$ again, we find the inequality
$$
\aligned
0&<\int_{\Omega}\int_{\Omega}\frac{|v_n(y)||v_n(x)|}{|x-y|^{\mu}}\Big[\frac{F((1+\delta)v_n(y))f((1+\delta)v_n(x))(1+\delta)v_n(x)}{(1+\delta)|v_n(y)|(1+\delta)|v_n(x)|}\\
&\qquad-\frac{F(v_n(y))f(v_n(x))v_n(x)}{|v_n(y)||v_n(x)|}\Big]\\
&=\int_{\Omega}\int_{\Omega}\Big[\frac{F((1+\delta)v_n(y))f((1+\delta)v_n(x))(1+\delta)v_n(x)}{(1+\delta)^2|x-y|^{\mu}}-\frac{F(v_n(y))f(v_n(x))v_n(x)}{|x-y|^{\mu}}\Big]\\
&\leq\zeta C+o_n(1)
\endaligned
$$
Letting $n\to \infty$ in the last inequality and applying Fatou's lemma, it follows
that
$$
0<\int_{\Omega}\int_{\Omega}\frac{F((1+\delta)v(y))f((1+\delta)v(x))(1+\delta)v(x)}{(1+\delta)^2|x-y|^{\mu}}-\frac{F(v(y))f(v(x))v(x)}{|x-y|^{\mu}}\leq\zeta C
$$
which is absurd, since the arbitrariness of $\zeta $.

\par
\noindent {\bf Step 2.} $t_0=1$. \\
\noindent In this case, we begin with recalling that $m_{V_{\infty}}\leq \Phi_{V_{\infty}}(t_nu_n)$. Therefore,
$$
c_{\vr}+o_n(1)=I_{\vr}(u_n) \geq I_{\vr}(u_n)+m_{V_{\infty}}-\Phi_{V_{\infty}}(t_nu_n).
$$
and from
$$
\aligned
I_{\vr}(u_n)-\Phi_{V_{\infty}}(t_nu_n)&=\frac{(1-t^2_n)}{2}\int_{\mathbb{R}^2} |\nabla u_n|^2+ \frac{1}{2}\int_{\mathbb{R}^2}V(\varepsilon x) |u_n|^2\\
&\hspace{4mm}-\frac{t^2_n}{2}\int_{\mathbb{R}^2}V_{\infty}|u_n|^2+\mathfrak{F}(t_nu_n)-\mathfrak{F}(u_n),
\endaligned
$$
and the fact that $\{u_n\}$ is bounded in $E_{\vr}$ as well as $u_n\rightharpoonup0$,  we derive from \eqref{V1}
$$
c_{\vr}+o_n(1) \geq m_{V_{\infty}}-\zeta C+o_n(1),
$$
and since $\zeta$ is arbitrary we obtain
$$
\limsup_{\vr \to 0}c_{\vr}\geq m_{V_{\infty}},
$$
which  contradicts Lemma \ref{EML}.
\vspace{0.2 cm}

\noindent {\bf Step 3.} $t_0<1.$ \\
\noindent  In this case, we may assume that $t_n <1$ for all $n \in\mathbb{N}$.  Since $m_{V_{\infty}}\leq \Phi_{V_{\infty}}(t_nu_n)$ and $\langle\Phi'_{V_{\infty}}(t_nu_n),t_nu_n\rangle=0$, we have
$$
\aligned
m_{V_{\infty}}&\leq \Phi_{V_{\infty}}(t_nu_n)-\frac12\langle \Phi_{V_{\infty}}'(t_nu_n),t_nu_n\rangle\\
&=\frac12\mathfrak{F}'(t_nu_n)[t_nu_n]-\mathfrak{F}(t_nu_n)\\
&=\frac12\int_{\R^2}\int_{\R^2}\frac{F(t_nu_n(y))f(t_nu_n(x))t_nu_n(x)}{|x-y|^{\mu}}-\frac12\int_{\R^2}\int_{\R^2}\frac{F(t_nu_n(y))F(t_nu_n(x))}{|x-y|^{\mu}}\\
&< \frac12\int_{\mathbb{R}^2}\int_{\mathbb{R}^2}\frac{F(u_n(y))f(u_n(x))u_n(x)}{|x-y|^{\mu}}-\frac12\int_{\mathbb{R}^2}\int_{\mathbb{R}^2}\frac{F(u_n(y))F(u_n(x))}{|x-y|^{\mu}}\\
&= I_{\vr}(u_n)-\frac12\langle I'_{\vr}(u_n),u_n\rangle\\
&=c_{\vr}+o_n(1),
\endaligned
$$
which yields a contradiction also in this case. From Steps 1, 2 and 3, we deduce that $u_\vr \not=0$. Hence, by Fatou's Lemma and using the characterization of $c_\vr$, it follows that
$$
\aligned
c_\vr&\leq I_\vr(u_\vr)
=I_\vr(u_\vr)-\frac{1}{2}\langle I'_\vr(u_\vr), u_\vr\rangle\\
&=\frac12\int_{\mathbb{R}^2}\int_{\mathbb{R}^2}\frac{F(u_\vr(y))[f(u_\vr(x))u_\vr(x)-F(u_\vr(x)]}{|x-y|^{\mu}}\\
&= \liminf_{n \to +\infty}\frac12\int_{\mathbb{R}^2}\int_{\mathbb{R}^2}\frac{F(u_n(y))[f(u_n(x))u_n(x)-F(u_n(x)]}{|x-y|^{\mu}}\\
 &\leq\limsup_{n \to +\infty}(I_\vr(u_n)-\frac{1}{2}\langle I'_\vr(u_n), u_n\rangle)=c_\vr
\endaligned
$$
thus
$$
I_\vr(u_\vr)=c_\vr.
$$
Now, using the following inequalities
$$
c_\vr=I_\vr(u_\vr)-\frac{1}{2K}\langle I'_\vr(u_\vr), u_\vr\rangle \leq \liminf_{n \to +\infty}(I_\vr(u_n)-\frac{1}{2K}\langle I'_\vr(u_n), u_n\rangle)\leq \limsup_{n \to +\infty}(I_\vr(u_n)-\frac{1}{2K}\langle I'_\vr(u_n), u_n\rangle)=c_\vr
$$
we actually have
$$
u_n \to u_\vr \quad \mbox{in} \quad E_\vr,
$$
showing that $I_\vr$ verifies the $(PS)_{c_\vr}$ condition.
\end{proof}
As an immediate consequence of Lemma \ref{PS}, we have
\begin{cor}\label{Existence}
The minimax value $c_{\vr}$ is achieved if $\vr$ is small enough and hence problem $(SNS^*)$ has a solution of least energy if  $\vr$ is small enough.
\end{cor}
\section{Concentration phenomena: proof of Theorem \ref{T1} completed}

In this section our goal is to establish the concentration phenomenon fro ground state solutions of the singularly perturbed equation $(SNS^*)$ . For this purpose, the following technical lemma will play a fundamental role.
\begin{lem}\label{BNT1}
Suppose that assumptions $(f_1)$ and $(f_2)$ hold. If $h \in H^{1}(\mathbb{R}^{2})$, then the function
$
\frac{1}{|x|^{\mu}}\ast F(h)
$
belongs to $L^{\infty}(\mathbb{R}^{2})$.
\end{lem}
\begin{proof}
For $\beta>1$, there exists $C_0>0$ such that
$$
F(s)\leq C_0\Big( |s|^{\frac{4-\mu}{2}}+|s|\big[ e^{\beta4\pi s ^{2}}
-1 \big]\Big), \forall s\in \R.
$$
Then,
$$
\aligned
\big|\frac{1}{|x|^{\mu}}\ast F(h)\big|&=\Big|\int_{\R^2}\frac{F(h)}{|x-y|^\mu}\Big|\\
&=\Big|\int_{|x-y|\leq1}\frac{F(h)}{|x-y|^\mu}\Big|+C\Big|\int_{|x-y|\geq1}\frac{F(h)}{|x-y|^\mu}\Big|\\
&\leq \int_{|x-y|\leq1}\frac{|h|^{\frac{4-\mu}{2}}+|h|\big[ e^{\beta4\pi |h|^{2}}
-1 \big]}{|x-y|^\mu}\\
&\hspace{5mm}+C\int_{|x-y|\geq1}\Big(\frac{|h|^{\frac{4-\mu}{2}}}{|x-y|^{\mu}}+|h|\big[ e^{\beta4\pi |h|^{2}}
-1 \big]\Big).
\endaligned
$$
Since
$$
\frac{1}{|y|^{\mu}} \in L^{\frac{2+\delta}{\mu}}(B_1^{c}(0)), \,\,\,\, \forall ~~ \delta >0,
$$
take $\delta \approx 0^{+}$ such that
$$
q_{1,\delta}=\frac{(4-\mu)}{2}\frac{(2+\delta)}{(2+\delta)-\mu}>2.
$$
Using H\"{o}lder inequality, we get
$$
\int_{|x-y|\geq1}\frac{|h|^{\frac{4-\mu}{2}}}{|x-y|^\mu} \leq C_0\left(\int_{|x-y|\geq 1}|h|^{q_{1,\delta}} \right)^{\frac{(2+\delta)-\mu}{2+\delta}}=C_1.
$$
On the other hand,  by Lemma \ref{Trudinger-Moser}

$$
e^{2\beta 4 s\pi |h|^{2}}-1 \in L^{1}(\mathbb{R}^{2}), \quad \forall s \geq 1,
$$

\noindent Again by H\"older's inequality
$$
\int_{|x-y|\geq1}|h|\big[ e^{\beta4\pi |h|^{2}}
-1 \big]\leq |h|_2\int_{\R^2}\Big(\big[ e^{2\beta 4\pi\frac{ |h|^{2}}{\|h\|_{\vr}^2}}
-1 \big]\Big)^{\frac12}\leq C_2.
$$
for some positive constant $C_2$.

\noindent Choosing  $t\in (\frac{2}{2-\mu}, +\infty)$, we have that $\frac{(4-\mu)t}{2}>2$ and $1-\frac{t\mu}{t-1}>-1$. Then, from H\"{o}lder's inequality
$$
\aligned
\int_{|x-y|\leq1}\frac{|h|^{\frac{4-\mu}{2}}}{|x-y|^\mu}&\leq \left(\int_{|x-y|\leq1}|h|^{\frac{(4-\mu)t}{2}}\right)^{\frac1t}\left(\int_{|x-y|\leq1}\frac{1}{|x-y|^{\frac{t\mu}{t-1}}}\right)^{\frac{t-1}{t}}\\
&\leq  C_2\left(\int_{|r|\leq1}{|r|^{1-\frac{t\mu}{t-1}}}dr\right)^{\frac{t-1}{t}} =C_3.\\
\endaligned
$$
Furthermore,  using again Lemma \ref{Trudinger-Moser}, we get
$$\aligned
\int_{|x-y|\leq1}&\frac{|h|\big[ e^{\beta4\pi |h|^{2}}
-1 \big]}{|x-y|^\mu}\\
&\leq \big(\int_{|x-y|\leq1}|h|\big[ e^{\beta4\pi |h|^{2}}
-1 \big]|^{t}\big)^{\frac1t}\big(\int_{|x-y|\leq1}\frac{1}{|x-y|^{\frac{t\mu}{t-1}}}\big)^{\frac{t-1}{t}}\\
&\leq  \big(\int_{|x-y|\leq1}|h|^{2t}\big)^{\frac{1}{2t}}\big(\int_{|x-y|\leq1}\big[ e^{ 2\beta t 4\pi |h|^{2}}
-1 \big]\big)^{\frac{1}{2t}}\big(\int_{|r|\leq1}{|r|^{1-\frac{t\mu}{t-1}}}dr\big)^{\frac{t-1}{t}}\\
&\leq C_4.
\endaligned
$$
\noindent Joining the above estimates the lemma follows.
\end{proof}

\begin{Prop}\label{Seq}
Let $\vr_n\to0$ and $\{u_n\}$ be the sequence of solutions obtained in Corollary \ref{Existence}.
Then, there exists a sequence $\{{y}_n\}\subset \R^2$, such that $v_n=u_n(x+{y}_n)$ has a convergent subsequence in $E$. Moreover, up to a subsequence, $y_n\to y\in M$.
\end{Prop}
\begin{proof}
Let $\{u_n\}$ be the sequence of solutions obtained in Corollary \ref{Existence}, it is easy to see $c_{\vr_n}=I_{\vr_n}(u_n)\to m_{V_0}$, $\{u_n\}$ is bounded in $E$ and
$$
0<m_{V_0} = \limsup_{n\to\infty}c_{\vr_n}<\frac{(4-\mu)}{8}.
$$
By following the argument in the proof of Theorem \ref{thm-Existence} in Section 2, there exist
$r, \delta>0$ and $\tilde{y}_n\in \R^2$ such that
\begin{equation} \label{B1'}
\liminf_{n\to\infty}\int_{B_r(\tilde{y}_n)}|u_n|^2\geq\delta.
\end{equation}
  Setting  $v_{n}(x)=u_{n}(x+\tilde{y}_{n})$, up to a subsequence, if necessary, we may assume  $v_{n}\rightharpoonup v\not\equiv0$ in $E$. Let $t_{n}>0$ be such that $\tilde{v}_{n}=t_{n}v_{n}\in \mathcal{N}_{V_{0}}$. Then,
$$
m_{V_0}\leq \Phi_{V_{0}}(\tilde{v}_{n})=\Phi_{V_{0}}(t_{n}u_{n})\leq I_{\vr}(t_{n}u_{n})\leq I_{\vr}(u_{n})\to m_{V_0}
$$
and so,
$$
\Phi_{V_{0}}(\tilde{v}_{n})\rightarrow m_{V_0}\ \hbox{and}\  (\tilde{v}_{n})\subset\mathcal{N}_{V_{0}}.
$$
Then the sequence $\{\tilde{v}_{n}\}$ is a minimizing sequence, and by the Ekeland Variational Principle \cite{E}, we may also assume it is a bounded $(PS)$ sequence at $m_{V_0}$. Thus, for some subsequence, $\tilde{v}_{n}\rightharpoonup \tilde{v}$ weakly in $E$ with $\tilde{v}\neq 0$ and $\Phi'_{V_{0}}(\tilde{v})=0$. Repeating the same arguments used in the proof of Lemma \ref{PS}, we have that $\tilde{v}_{n}\rightarrow \tilde{v}\ \ \hbox{in}\ \ E.$ Since $(t_n)$ is bounded, we can assume that for some subsequence $t_{n}\rightarrow t_{0}>0$, and so $v_{n}\rightarrow v$ in $E$.

\noindent Next we will show that $\{y_{n}\}=\{\vr_{n}\tilde{y}_{n}\}$ has a subsequence satisfying $y_{n}\rightarrow y\in M$. We begin with proving that
$\{y_{n}\}$ is bounded in $\mathbb{R}^2$. Indeed, if not there would exist a subsequence, which we still denote by $\{y_{n}\}$, such that $|y_{n}|\rightarrow\infty$. Since $\tilde{v}_{n}\rightarrow \tilde{v}$ in $E$ and $V_{0}<V_{\infty}$, we have
$$
\aligned
m_{V_{0}}&=\frac{1}{2}\int_{\mathbb{R}^{2}}|\nabla \tilde{v}|^{2}+\frac{1}{2}\int_{\mathbb{R}^{2}}V_{0}| \tilde{v}|^{2}-\mathfrak{F}(\tilde{v})\\
&<\frac{1}{2}\int_{\mathbb{R}^{2}}|\nabla \tilde{v}|^{2}+\frac{1}{2}\int_{\mathbb{R}^{2}}V_{\infty}| \tilde{v}|^{2}-\mathfrak{F}(\tilde{v})\\
&\leq\liminf_{n\rightarrow\infty}\left[\frac{1}{2}\int_{\mathbb{R}^{2}}|\nabla \tilde{v}_{n}|^{2}+\frac{1}{2}\int_{\mathbb{R}^{2}}V(\epsilon_{n}x+y_{n})| \tilde{v}_{n}|^{2}-\mathfrak{F}(\tilde{v}_{n})\right]\\
&=\liminf_{n\rightarrow\infty}\left[\frac{t_{n}^{2}}{2}\int_{\mathbb{R}^{2}}|\nabla u_{n}|^{2}+\frac{t_{n}^{2}}{2}\int_{\mathbb{R}^{2}}V(\epsilon_{n}x)| u_{n}|^{2}-\mathfrak{F}(t_{n}^{2}u_{n})\right]\\
&=\liminf_{n\rightarrow\infty}I_{\vr_n}(t_nu_n)\\
&\leq \liminf_{n\rightarrow\infty}I_{\vr_n}(u_n)\\
&=m_{V_{0}}
\endaligned
$$
hence the absurd which shows that $\{y_{n}\}$ stays bounded and up to a subsequence, $y_n\to y\in\R^2$. Then, necessarily  $y\in M$ otherwise we would get again a contradiction as above.
\end{proof}

\noindent Let $\vr_n\to 0$ as $n\to \infty$, $u_n$ be the ground state solution of
$$
\aligned &-\Delta u +V(\vr_nx)u  =\Big[\frac{1}{|x|^{\mu}}\ast  F(u)\Big]f(u) \quad \mbox{in} \quad \R^2 .
\endaligned
 $$
From Lemma \ref{EML} we know
$$
 I_{\vr_n}(u_n)\to m_{V_0}.
$$
Then, there exists a sequence $\tilde{{y}}_n\in \R^2$, such that $v_n=u_n(x+\tilde{{y}}_n)$ is a solution of
$$
-\Delta v_n +V_n(x)v_n=\Big[\frac{1}{|x|^{\mu}}\ast  F(v_n)\Big]f(v_n), \quad \mbox{in} \quad \R^2,
$$
where $V_{n}(x)=V(\vr_{n}x+\vr_{n}\tilde{{y}}_{n})$. Moreover, $(v_n)$ has a convergent subsequence in $E$ and  ${y}_n\to y\in M$, up to a subsequence, where ${y}_n=\vr_n \tilde{{y}}_n$. Hence, there exists $h \in H^{1}(\mathbb{R}^{2})$ such that
\begin{equation} \label{h}
|v_n(x)| \leq h(x) \quad \mbox{a.e in} \quad \mathbb{R}^{2} \quad \forall n \in \mathbb{N}.
\end{equation}

\begin{lem}
 Suppose that conditions $(f_1)-(f_5)$, $(V_1)$ and $(V_2)$ hold. Then there exists $C>0$ such that $\|v_{n}\|_{L^{\infty}(\mathbb{R}^2)}\leq C$ for all $n\in \mathbb{N}$. Furthermore
$$
\lim_{|x|\rightarrow \infty}v_{n}(x)=0\ \hbox{uniformly in}\ n\in \mathbb{N}.
$$
\end{lem}
\begin{proof}
Let us first show that the sequence
$$
W_n(x):= \Big[\frac{1}{|x|^{\mu}}\ast  F(v_n)\Big],
$$
stays bounded in $L^{\infty}(\mathbb{R}^{2})$. Indeed, as $F$ is an increasing function, by $(\ref{h})$ we know that
$$
0 \leq W_n(x):= \Big[\frac{1}{|x|^{\mu}}\ast  F(v_n)\Big] \leq  \Big[\frac{1}{|x|^{\mu}}\ast  F(h)\Big]
$$
Hence claim will hold provided the function
$$
W(x)=\Big[\frac{1}{|x|^{\mu}}\ast  F(h)\Big]
$$
belongs to $L^{\infty}(\mathbb{R}^{2})$ and this is an immediate consequence of Lemma \ref{BNT1}.

\noindent For any $R>0$, $0<r\leq\frac{R}{2}$, let $\eta\in C^{\infty}(\mathbb{R}^2)$, $0\leq\eta\leq1$ with $\eta(x)=1$ if $|x|\geq R$ and $\eta(x)=0$ if $|x|\leq R-r$ and $|\nabla\eta|\leq\frac{2}{r}$.
For $L>0$, let
$$
v_{L,n}=
\left\{\begin{array}{l}
\displaystyle v_n(x),
\hspace{1.14mm} v(x)\leq L\\
\displaystyle L,\hspace{8.14mm}v_n(x)\geq L,
\end{array}
\right.
$$
and
$$
z_{L,n}=\eta^{2}v_{L,n}^{2(\gamma-1)}v_n\hspace{3.14mm}and\hspace{3.14mm}w_{L,n}=\eta v_nv_{L,n}^{\gamma-1}
$$
with $\gamma>1$ to be determined later. Taking $z_{L,n}$ as a test function, we obtain
\begin{equation}\label{E1}
\aligned
\int_{\mathbb{R}^{2}}&\eta^{2}v_{L,n}^{2(\gamma-1)}|\nabla v_n|^{2}+\int_{\mathbb{R}^{2}}\tilde{V}_{\vr_n}(x)| v_n|^{2}\eta^{2}v_{L,n}^{2(\gamma-1)}\\
&=-2(\gamma-1)\int_{\mathbb{R}^{2}}v_nv_{L,n}^{2\gamma-3}\eta^{2}\nabla v_n\nabla v_{L,n}+\int_{\mathbb{R}^{2}}W_n(x)f(v_n)\eta^{2}v_nv_{L,n}^{2(\gamma-1)}\\
&\hspace{1cm}-2\int_{\mathbb{R}^{2}}\eta v_{L,n}^{2(\gamma-1)}v_n\nabla v_n\nabla \eta.
\endaligned
\end{equation}
Using Lemma \ref{Trudinger-Moser}, for all $\beta, s>1$, we know that
\begin{equation}\label{E2}
\int_{\mathbb{R}^2}\big[ e^{ \beta4\pi v_n ^{2}} -1 \big]^s\leq \int_{\mathbb{R}^2}\big[ e^{ \beta4\pi |h|^{2}} -1 \big]^s=C < \infty \,\,\,\, \forall n \in \mathbb{N}.
\end{equation}
Let $t=\sqrt{s}$, $p>\frac{2t}{t-1}>2$ and $\gamma=\frac{p(t-1)}{2t}$, for any $\delta>0$, there exists $C(\delta, p, \beta)>0$ such that
$$
F(u)\leq \delta u^{2}+C(\delta, p, \beta) u^{p-1}\big[ e^{ \beta4\pi |u| ^{2}}
-1 \big], \,\,\, \forall u\in \R.
$$
Thus for $\delta$ sufficiently small, as $(W_n)$ is bounded in $L^{\infty}(\mathbb{R}^{2})$, gathering \eqref{E1} and Young's inequality, we get
\begin{equation}\label{E3}
\aligned
\int_{\mathbb{R}^{2}}\eta^{2}&v_{L,n}^{2(\gamma-1)}|\nabla v_n|^{2}+V_0\int_{\mathbb{R}^{2}}| v_n|^{2}\eta^{2}v_{L,n}^{2(\gamma-1)}\\
&\leq C\int_{\mathbb{R}^{2}} v_n^{p}\eta^{2}v_{L,n}^{2(\gamma-1)}\big[ e^{ \beta4\pi |h|^{2}}
-1 \big]+C\int_{\mathbb{R}^{2}}v^2_nv_{L,n}^{2(\gamma-1)}|\nabla \eta|^2.
\endaligned
\end{equation}
Using this fact, from \cite{AF} we have
$$
|w_{L,n}|^2_p\leq C\gamma^2\Big(C'+\Big[\int_{|x|\geq R-r}{v_n}^{(p-2)t}\big[ e^{ \beta4\pi |h|^{2}}
-1 \big]^t\Big]^{\frac1t}\Big)\Big[\int_{|x|\geq R-r}{{v_n}^{{\frac{2\gamma t}{t-1}}}}\Big]^{{\frac{t-1}{t}}}.
$$
By \eqref{E2} and H\"{o}lder's inequality, we know
$$
|w_{L,n}|^2_p\leq C\gamma^2\Big[\int_{|x|\geq R-r}{{v_n}^{{\frac{2\gamma t}{t-1}}}}\Big]^{{\frac{t-1}{t}}}.
$$
Now, following the same iteration arguments explored in \cite{AF}, we find
\begin{equation}\label{BD1}
|v_n|_{L^{\infty}(|x|\geq R)}\leq C|v_n|_{{p}(|x|\geq R/2)}.
\end{equation}
For $x_0\in B_R$, we can use the same argument taking $\eta\in C^{\infty}_0(\R^2,[0,1])$ with $\eta(x)=1$ if $|x-x_0|\leq\rho'$ and $\eta(x)=0$ if $|x-x_0|>2\rho'$ and $|\nabla\eta|\leq \frac{2}{\rho'}$, to prove that
\begin{equation}\label{BD2}
|v_n|_{L^{\infty}(|x-x_0|\leq \rho')}\leq C|v_n|_{{p}(|x|\leq2\rho')}.
\end{equation}
With \eqref{BD1} and \eqref{BD2}, by a standard covering argument it follows that
$$
|v_n|_{{\infty}}<C
$$
for some positive constant $C$. Then, using again the convergence of $(v_{n})$ to $v$ in $E$ in the right side of \eqref{BD1}, for each $\delta> 0$ fixed, there exists $R>0$ such that $|v_{n}|_{L^{\infty}(|x|\geq R)}<\delta, \forall n\in N.$
Thus,
$$
\lim\limits_{\mid x\mid\rightarrow\infty}{v_{n}}(x)=0 \quad \mbox{uniformly in} \quad n \in \mathbb{N},
$$
and the proof is complete.
\end{proof}

The last lemma establishes an estimate from below in terms of the $L^\infty$-norm of $\{v_n\}$.

\begin{lem}\label{MP}
There exists $\delta_0 >0$ such that $|v_{n}|_{\infty}\geq \delta_0$ for all $n\in \mathbb{N}$.
\end{lem}
\begin{proof}
Recall that,
$$
 \delta \leq \int_{B_{r}(\tilde{y}_n)}|u_n|^{2},
$$
then
$$
\delta\leq\int_{B_{r}(0)}|v_n|^{2} \leq |B_r||v_n|^2_{\infty},
$$
from where it follows
$$
|v_n|_{\infty} \geq \delta_0,
$$
showing the lemma. \end{proof}

\vspace{0.5 cm}

{\bf{Concentration around  maxima.}}  Let $b_n$ denote a maximum point of $v_n$, we know it is a bounded sequence in $\R^2$. Thus, there is $R>0$ such that
$b_n\in B_{R}(0)$. Thus  the global maximum of $u_{{\vr}_{n}}$ is attained at $z_{n}=b_n+\tilde{y}_n$ and
$$
\vr_{n}{z}_{n}=\vr_{n}b_{n}+\vr_{n}\tilde{y}_n=\vr_{n}b_{n}+y_{n}.
$$
From the boundedness of $\{b_n\}$ we have
$$
\lim_{n \to \infty}z_{n}=y,
$$
which together with the continuity of $V$ yields
$$
\lim\limits_{ n \rightarrow\infty}V(\vr_{n}z_{n})=V_{0}.
$$

\noindent If $u_{\vr}$ is a positive solution of $(SNS^*)$ the function ${w_\vr}(x)={u_\vr}(\frac{x}{\vr})$
is a positive solution of \eqref{EC}. Thus, the maxima points ${\eta_\vr}$ and ${z_\vr}$ of respectively ${w_\vr}$ and ${u_\vr}$,
satisfy the equality ${\eta_\vr}=\vr{z_\vr}$ and in turn
$$
\lim_{\vr \rightarrow 0}V(\eta_{\vr})=V_{0}.
$$

\end{document}